\documentclass{amsart}

\pdfoutput=1
\usepackage[all,hyperref,numberbysection,Rn]{paper_diening}
\numberwithin{equation}{section}

\usepackage{tikz}
\usetikzlibrary{calc}

\usepackage{mathtools}
\usepackage[tworuled]{algorithm2e} 
\usepackage[thinlines]{easytable}

\providecommand{\vertices}{\ensuremath{\mathcal{N}}}

\providecommand{\tria}{\ensuremath{\mathcal{T}}}

\providecommand{\CBY}{C^\textup{BY}}
\providecommand{\CCR}{C_\textup{CR}}
\providecommand{\QCR}{Q_\textup{CR}}
\providecommand{\deltaCR}{\delta_\textup{CR}}
\providecommand{\CR}{\textup{CR}^1(\mathcal{T})}
\providecommand{\PiSZ}{\ensuremath{\Pi_{\textup{SZ}}}}

\providecommand{\dist}{\mathrm{dist}}

\providecommand{\Cloc}{C^{\loc}}

\providecommand{\nablaNC}{\nabla_\textup{NC}}
\providecommand{\INC}{\mathcal{I}_\textup{NC}}

 \providecommand{\highlight}[1]{\mathbf{#1}}

\providecommand{\dx}{\,\mathrm{d}x}

\begin{document}

\author[L. Diening]{Lars Diening}
\author[J. Storn]{Johannes Storn}
\author[T. Tscherpel]{Tabea Tscherpel}

\address[L. Diening, J. Storn, Tabea Tscherpel]{Department of Mathematics, University of Bielefeld, Postfach 10 01 31, 33501 Bielefeld, Germany}
\email{lars.diening@uni-bielefeld.de}
\email{jstorn@math.uni-bielefeld.de}
\email{ttscherpel@math.uni-bielefeld.de}

\thanks{This research was supported by the DFG through the CRC 1283 ``Taming uncertainty and profiting from randomness and low regularity in analysis, stochastics and their applications"}

\subjclass[2010]{
65N30, 		
 65N50,	 	
 65N12, 		
65M60}		

\keywords{$L^2$-projection, $L^{p}$-stability, Sobolev stability, adaptive mesh refinement, Lagrange elements, Crouzeix--Raviart elements}

\title{On the Sobolev and $L^p$-Stability of the $L^2$-Projection}

\begin{abstract}
We show stability of the $L^2$-projection onto Lagrange finite element spaces with respect to (weighted) $L^p$ and $W^{1,p}$-norms for any polynomial degree and for any space dimension under suitable conditions on the mesh grading. 
This includes $W^{1,2}$-stability in two space dimensions for any polynomial degree and meshes generated by newest vertex bisection. 
Under realistic but conjectured assumptions on the mesh grading in three dimensions we show $W^{1,2}$-stability for all polynomial degrees. 
We also propose a modified bisection strategy that leads to better $W^{1,p}$-stability. 
Moreover, we investigate the stability of the $L^2$-projection onto Crouzeix--Raviart elements.   
\end{abstract}

\maketitle 

\setcounter{tocdepth}{3}

\section{Introduction}
\label{sec:intro}

The stability of the $L^2$-projection onto finite element spaces is of significant importance in the analysis of finite element methods. 
More specifically, it is essential in the analysis of parabolic problems, e.g.~\cite{ErikssonJohnson95}. 
For instance it has been shown in~\cite{TV.2016} that the optimality of numerical schemes for linear parabolic problems is equivalent to $W^{1,2}$-stability of the $L^2$-projection. 
Sobolev stability of the $L^2$-projection is also crucial in the study of non-linear parabolic problems~\cite{BreDieStoWic20} in order to avoid a coupling between the temporal and the spatial resolution.

For quasi-uniform meshes $W^{1,2}$-stability follows directly by inverse estimates and by the properties of the Scott--Zhang interpolation operator, cf.~\cite{BrambleXu91}. 
For general graded meshes the situation is more delicate: In~\cite[Sec.\
7]{BanYse14} Bank and Yserentant present a one dimensional counterexample to $W^{1,2}$-stability using strongly graded meshes. 
Thus, $W^{1,2}$-stability cannot hold without restrictions on the mesh. 

It is possible to show $W^{1,2}$-stability if the mesh size varies slightly~ \cite{CrouzeixThomee87,ErikssonJohnson95,Boman06}. 
However, this is not satisfied for adaptive refinement strategies in practice. 
In the following we focus on highly graded meshes that occur in adaptive finite element methods. 

To show $W^{1,2}$-stability for highly graded meshes there are two approaches in the literature. 
The first one achieves stability of the $L^2$-projection with respect to weighted $L^2$-norms by investigation of locally defined weighted mass matrices \cite{BramblePasciakSteinbach02,Carstensen02,GaspozHeineSiebert19pre,Steinbach01,Steinbach02}.
The second approach involves estimating the spatial decay of the $L^2$-projection and then deriving stability with respect to weighted $L^2$-norms \cite{CrouzeixThomee87,BanYse14}. 
Provided that the mesh size does not change too fast the weighted $L^2$-stability implies $W^{1,2}$-stability.

A measure for the change of the mesh size is the so-called grading of the mesh. 
In \cite{C.2004,GHS.2016,KarkulikPavlicekPraetorius13} the grading of meshes resulting from some adaptive refinement strategies is investigated. 
In certain cases this allows to verify the $W^{1,2}$-stability for adaptively refined meshes. 
For example the combination of \cite{BramblePasciakSteinbach02,C.2004} shows $W^{1,2}$-stability of the $L^2$-projection for linear elements and meshes resulting from red-green-blue
refinement. 
In~\cite{BanYse14,GHS.2016} the $W^{1,2}$-stability is proved for polynomials up to order twelve for 2D meshes generated by the newest vertex bisection. 
In 3D under reasonable assumptions on the grading for a (higher dimensional) bisection algorithm, as stated in Conjecture~\ref{conj:NVB-grad}, ~\cite{BanYse14} prove $W^{1,2}$-stability for polynomials up to order seven.
In Section~\ref{sec:results-diff-refin} we give a more detailed review on existing results concerning $W^{1,2}$-stability for various refinement strategies. 

Here we utilize ideas of Bank and Yserentant \cite{BanYse14} to derive improved decay estimates for the $L^2$-projection. 
We achieve this by approximation of the $L^2$-projection with a novel, alternative operator $C$, which is constructed of local weighted projections. 
In contrast to~\cite{BanYse14,GaspozHeineSiebert19pre} our analysis of the operator $C$ avoids numerical calculation of eigenvalues. 
Thus, the results apply to \textit{all} dimensions $d\in \mathbb{N}$ and polynomial degrees $K\in \mathbb{N}$. 

Then, the decay estimates yield weighted $L^2$-stability. 
In fact those two concepts are equivalent, see Remark~\ref{rem:equivalence-decay-weighted}. 
We use the weighted $L^2$-stability to conclude (weighted) $L^p$ and $W^{1,p}$-stability in Theorem \ref{thm:weighted-Lp} and \ref{thm:weighted-Sobolev}. 
A crucial tool in the derivation is a maximal operator $M_\gamma$ introduced in Definition \ref{def:M-graded}. 
Since the proofs solely rely on the weighted $L^2$-estimates, the result is independent of our proof of the decay estimate. 
Hence, it is directly applicable to improved weighted $L^2$-estimates in prospective works. 
Notice that Crouzeix and Thom\'ee achieve similar results for the lowest-order case $K=1$ under an assumption on the number of simplices in a 'distance' layer starting from an arbitrary simplex \cite{CrouzeixThomee87}. 
This additional assumption restricts the admissible meshes for $W^{1,p}$-stability significantly, cf.\ Remark \ref{rem:CR}, and is avoided in our analysis. 
Moreover, we do not use interpolation with $L^{1}$ and $L^{\infty}$ that would require for $W^{1,p}$-stability the strongest assumption on the mesh grading corresponding to the $W^{1,\infty}${-}case. 
Hence, our assumptions on the mesh grading depend continuously on the Lebesgue exponents. 
This avoids the mentioned disadvantages of the analysis in \cite{Boman06,ErikssonJohnson95}. 
Our proofs directly extend to stability results for the $L^2$-projection onto Lagrange spaces equipped with zero boundary values.

Overall, our investigation verifies $L^p$ and $W^{1,p}$-stability for meshes generated by different refinement strategies depending on the polynomial degrees $K$ and Lebesgue exponents $p \in [1,\infty]$.
Tables \ref{tab:grading-stability-2D} and \ref{tab:grading-stability-3D} display admissible ranges of the parameters $K$ and $p$. 
In two dimensions we show $W^{1,2}$-stability for \textit{all} polynomial degrees $K \geq 1$ for meshes generated by newest vertex bisection with the grading estimate obtained in \cite{GHS.2016}.
For the same meshes we are the first to prove $L^1$ and $L^\infty$-stability for all $K\geq 1$ and $W^{1,1}$ and $W^{1,\infty}$-stability for all $K\geq 3$. 
This is an important tool in the numerical analysis of the $p$-heat equation, cf.\ \cite{BreDieStoWic20}.
In three dimensions there exist no grading estimates for the  bisection algorithm by Maubach and Traxler \cite{Maubach95,Traxler97}. 
Under realistic assumptions on the grading conjectured in Section~\ref{subsec:RefStrat}
 we show $W^{1,2}$-stability for all polynomial degrees $K \in \mathbb{N}$. 

Since the grading for the higher dimensional bisection algorithm specified in Conjecture~\ref{conj:NVB-grad} has not been proved yet we introduce a modified adaptive refinement strategy with guaranteed grading.
The routine reduces the mesh grading and thus also leads to improved $L^p$ and $W^{1,p}$-stability while preserving the beneficial adaptive properties of the newest vertex bisection. 
In particular, it allows for optimal convergence of adaptive finite element schemes. 
A similar algorithm is proposed in \cite{DemlowStevenson2011} which has the same structure but is based on a different notion of mesh grading. 

For Crouzeix--Raviart finite elements a further alternative notion of mesh grading is used. 
This notion is natural for these finite elements and results in better grading estimates. 
Those estimates allow to conclude $L^p$ and (broken) $W^{1,p}$-stability of the $L^2$-projection onto these elements for any dimension $d \leq 32$ for the bisection algorithm mentioned above without modification.

Let us conclude with the structure of this paper. 
Section~\ref{sec:DecayWithC} describes the decay estimate of Bank and Yserentant \cite{BanYse14} in an abstract framework. 
In this setting decay estimates for the $L^2$-projection follow from the existence of an approximating operator $C$ analyzed in Section~\ref{dec:design-C}. 
In Section~\ref{sec:WeightedEst} we utilizes the decay estimate to deduce (weighted) $L^p$ and $W^{1,p}$-stability.
In Section~\ref{sec:results-diff-refin} we summarize available grading estimates and give an overview of the resulting stability estimates. Moreover, we introduce a modified refinement strategy with improved grading estimates.
This paper concludes with an investigation of the stability of the $L^2$-projection onto the space of Crouzeix--Raviart elements in Section \ref{sec:BY_CR}.

\section{The Decay Estimate of Bank and Yserentant}
\label{sec:DecayWithC}
This section introduces the framework in which the approach of Bank and Yserentant \cite{BanYse14} is set. 
This is the starting point of our investigation.

\subsection{Notation}
Let us start with some notation used throughout this paper. 
Let $\Omega$ be a bounded, polyhedral domain in~$\Rd$, for $d\in \mathbb{N}$. 
As usual $L^p(\Omega)$ and $W^{1,p}(\Omega)$ denote the Lebesgue and Sobolev spaces for $p\in[1,\infty]$ with norms $\norm{\cdot}_p$ and $\norm{\cdot}_{1,p}$. 
We denote by $\tria$ a regular triangulation of~$\Omega$ into closed simplices. 
By $\vertices= \vertices(\tria)$ we denote the set of vertices of~$\tria$. 
For $i \in \vertices$ let $\omega_i$ be the patch around~$i$, i.e., $\omega_i = \bigcup \set{T \in \mathcal{T}\colon i \in T}$. 
With slight abuse of notation we use~$\omega_i$ also for the collection of simplices in the patch, i.e.,  $\omega_i = \set{T \in \mathcal{T}\colon i \in T}$. 
For a simplex $T\in \mathcal{T}$ the space of polynomials of maximal degree $K\in \mathbb{N}$ reads $\mathcal{L}_K(T)$. 
For $K\in \mathbb{N}$ the space of piecewise polynomials and the space of Lagrange finite elements read
\begin{align*}
  \mathcal{L}^0_K(\tria) &\coloneqq \lbrace v \in L^\infty(\Omega)\colon v|_T \in \mathcal{L}_K(T) \text{ for all }T\in \mathcal{T}\rbrace,
  \\
  \mathcal{L}^1_K(\tria) &\coloneqq \lbrace v \in W^{1,\infty}(\Omega) \colon v\in \mathcal{L}_K^0(\tria)\rbrace. 
\end{align*}
For subdomains $\omega \subset \Omega$ we use $\norm{\cdot}_{p,\omega}$ for the norm of~$L^p(\omega)$.  
The inner product of~$L^2(\Omega)$ and $L^2(\omega)$ is denoted by $\langle \cdot,\cdot \rangle$ and $\langle \cdot,\cdot\rangle_\omega$, respectively. 
We denote the mean value integral by $\dashint_T \cdot\dx \coloneqq \frac{1}{\abs{T}} \int_T \cdot \dx$, where $\abs{T}$ is the volume of~$T \in \tria$. 

\subsection{The Approximating Operator~$C$}
Let $Q\colon L^2(\Omega)\to \mathcal{L}_{K}^1(\tria)$ denote the $L^2$-projection onto the Lagrange finite element space, that is,
\begin{align}
  \label{eq:defL2proj}
  \langle Qu,v_K\rangle = \langle u , v_K\rangle\qquad\text{for all
  }u\in L^2(\Omega), \;v_K\in \mathcal{L}^1_{K}(\tria). 
\end{align}
The operator $Q$ extends to an operator $Q \colon L^1(\Omega)\to \mathcal{L}_{K}^1(\tria)$. 

The approach by Bank and Yserentant~\cite{BanYse14} features an approximation of $Q$ by successive applications of a linear local operator~$C$. 
The properties of the operator~$C$ and a related distance~$\delta$ are summarized as follows.
\begin{enumerate}[label=(C\arabic{*})]
\item \label{itm:self-adjoint}
  \textbf{(Self-adjoint)} 
  The operator $C\colon L^2(\Omega) \to \mathcal{L}^1_K(\mathcal{T})$ is  self-adjoint and linear.
\item \label{itm:ellipticity}
 \textbf{(Ellipticity)} 
 The restriction $C|_{\mathcal{L}_K^1(\mathcal{T})}$ is $L^2$-elliptic with condition number 
  \begin{align*}
    \operatorname{cond}_2(C|_{\mathcal{L}_K^1(\mathcal{T})}) 
    =
    \frac{\lambda_{\max}(C|_{\mathcal{L}_K^1(\mathcal{T})})}{\lambda_{\min}(C|_{\mathcal{L}_K^1(\mathcal{T})})}<
    \infty. 
  \end{align*}
\item \label{itm:distance} 
\textbf{(Distance)} 
 A symmetric notion of neighbors for simplices $T,T' \in\tria$ with $T \neq T'$ is used which ensures (at least) that $T \cap T' \neq \emptyset$ if $T, T' \in \tria$ are neighbors. 
This induces an integer-valued metric~$\delta$ on $\tria$: 
For neighbors $T,T'$ we assign $\delta(T,T') = 1$. 
For any $T,T' \in \tria$  the value $\delta(T,T')$ is set to be the smallest integer $N \in \mathbb{N}_0$ such that there exist $T_0, \ldots, T_N \in \tria$ with $T_0 = T$, $T_N = T'$ and 
\begin{align*}
\delta(T_{j-1},T_j) = 1 \quad \text{ for all } j = 1, \ldots, N. 
\end{align*}
Since $\Omega$ is connected this means that $\delta$ is a geodesic distance. 
\item \label{itm:locality} 
\textbf{(Locality)} 
If $v \in L^2(\Omega)$ is supported on the collection~$L \subset \mathcal{T}$ of simplices, then the support of~$Cv$ is at most one layer of simplices larger, i.e., if $T' \subset \support(C v)$, then $\delta(T',L) \coloneqq \min_{T \in L} \delta(T',T) \leq 1$.
\end{enumerate}
Let us assume in the following that~$C$ and~$\delta$ satisfy~\ref{itm:self-adjoint}--\ref{itm:locality}.
Suitable operators $C$ and distance functions $\delta$ are introduced in Section \ref{dec:design-C} and \ref{sec:BY_CR}.

\begin{lemma}[Two-sided identity]
  \label{lem:two-sided-identity} We have  $C=QC =CQ$.
\end{lemma}
\begin{proof}
Since $Q|_{\mathcal{L}_K^1(\mathcal{T})}$ is the identity, one has that $C = QC$. 
This identity and the fact that both $Q$ and $C$ are self-adjoint imply that
  \begin{align*}
    & CQ = (Q^*C^*)^* = (QC)^* = C^* = C.
  \end{align*}
\end{proof}
We show that successively applying the operator $C$ yields an approximation of the $L^2$-projection $Q$. 
More precisely, let $\identity$ denote the identity mapping and set the function $u^{(0)} \coloneqq 0$. 
For a given function $u \in L^2(\Omega)$ we define $u^{(\nu+1)} \in \mathcal{L}_K^1(\mathcal{T})$ by the recursion
\begin{align}\label{eq:defRecursion}
  u^{(\nu+1)}-u  &= (\identity-C)(u^{(\nu)}-u)\qquad\text{for all }\nu \in\mathbb{N}.
\end{align} 
By the two-sided identity in Lemma~\ref{lem:two-sided-identity} we have $C=CQ$ which implies that
\begin{align*}
u - Qu = (\identity-C)(\identity-Q)u. 
\end{align*}
Applying this in the recursion \eqref{eq:defRecursion} yields that
\begin{align*}
  u^{(\nu+1)}- Qu
  &=   (\identity-C)(u^{(\nu)}-u) + (\identity-C)(\identity -Q)u
  = (\identity-C)(u^{(\nu)}-Qu).
\end{align*}
This identity proves that 
\begin{align}
  u^{(\nu)} - Qu  &= (\identity-C)^{\nu}(u^{(0)}-Qu).
\end{align}
Using optimal polynomials this process can be accelerated. 
We denote by~$p_{\nu}$ polynomials with coefficients $a_{l \nu} \in \mathbb{R}$ with $\sum_{l=0}^{\nu} a_{l \nu} = 1$ and 
\begin{align*}
  p_{\nu}(\lambda) &\coloneqq \sum_{l=0}^{\nu} a_{l \nu}\lambda^{l}\qquad\text{for all }\nu \in\mathbb{N}.
\end{align*}
We define the (accelerated) sequence~$Q^{(\nu)} u$ by $Q^{(0)} u \coloneqq 0$ and
\begin{align}
  \label{eq:Qnu}
  Q^{(\nu)} u&\coloneqq \sum_{l=0}^{\nu} a_{l \nu}u^{(l)}\qquad\text{for all }\nu \in\mathbb{N}.
\end{align}
Then we have that
\begin{align}
   Q^{(\nu)} u - Qu&= p_\nu(\identity -C) \, (u^{(0)}-Qu).
\end{align}
With the condition number $\kappa \coloneqq \operatorname{cond}_2(C|_{\mathcal{L}_K^1(\mathcal{T})})$ of~\ref{itm:ellipticity} and the decay parameter
  \begin{align}\label{eq:ConvSpeed2}
  q &\coloneqq \frac{\sqrt{\kappa}-1}{\sqrt{\kappa}+1},
  \end{align}
the theory of \emph{polynomially accelerated additive subspace correction} \cite[Lem.\ 2.3]{BanYse14} yields that optimal choice of coefficients $a_{l\nu}$ of the polynomials $p_{\nu}$ leads to the estimate
\begin{align}\label{eq:ConvSpeed}
  \norm{Q^{(\nu)} u - Qu}_2
  &\leq \frac{2 q^\nu}{1+q^{2\nu}} \,
    \norm{Qu}_2\qquad\text{for all }\nu \in \mathbb{N}.
\end{align}
Let $\indicator_A$ denote the indicator function for some
set $A \subset \Rd$. 
If $L$ is a collection of simplices, then we abbreviate $\indicator_L\coloneqq \indicator_{\bigcup L}$. 
The distance of two collections of simplices $L,L' \subset \mathcal{T}$ is defined by $\delta(L,L') \coloneqq \min \{ \delta(T,T')\colon T\in L,\, T' \in L'\}$. 

\begin{proposition}[{Decay estimate, \cite[Lem.~3.1]{BanYse14}}]\label{prop:decay}
  Let $L,L' \subset \mathcal{T}$ be collections of simplices and let $u\in L^2(\Omega)$. 
  Then
  \begin{align*}
    \norm{ \indicator_{L}Q (\indicator_{L'} u)}_2 \leq 
    \frac{2
    q^{\delta(L,L')-1}}{1+q^{2(\delta(L,L')-1)}}
    \norm{\indicator_{L'} u}_2 \leq
    \min \bigset{ 2
    q^{\delta(L,L')-1},1}
    \norm{\indicator_{L'} u}_2.
  \end{align*}
\end{proposition}
\begin{proof}
  Let us apply the accelerated polynomial subspace correction~\eqref{eq:Qnu} to the function $\indicator_{L'} u$, supported on~$L'$. 
  Due to the locality of~$C$ according to~\ref{itm:locality}, each application of~$C$ 
increases the support by at most one layer of simplices (measured in the distance $\delta$). 
  The same holds for the accelerated polynomial subspace
  correction~$Q^{(\nu)}(\indicator_{L'} u)$. 
  In particular, the support of~$Q^{(\nu)}(\indicator_{L'} u)$ is at most~$\nu$ layers larger than~$L'$. 
  Thus, $Q^{(\delta(L,L')-1)}(\indicator_{L'} u)$ is still zero on~$L$. 
  Hence,~\eqref{eq:ConvSpeed} implies that
  \begin{align*}
    \norm{ \indicator_{L}Q (\indicator_{L'} u)}_2
    &=
      \bignorm{ \indicator_{L}\big( Q (\indicator_{L'} u) -
      Q^{(\delta(L,L')-1)}(\indicator_{L'}u)\big)}_2
      \\
    &\leq
      \bignorm{Q (\indicator_{L'} u) - Q^{(\delta(L,L')-1)}(\indicator_{L'}u)}_2
     \\
     &\leq
      \frac{2 q^{\delta(L,L')-1}}{1+q^{2(\delta(L,L')-1)}}
      \norm{\indicator_{L'} u}_2. 
  \end{align*}
  The estimate $2t/(t^2+1)\leq 1$ for $t\in \mathbb{N}$ proves the second inequality. 
\end{proof}
\begin{remark}[Element-wise decay estimates]
  \label{rem:CR}
  Crouzeix and Thom{\'e}e~\cite{CrouzeixThomee87} also derive decay estimates for the $L^2$-projection. 
  However, they do not use layers but only $L=\set{T}$ and $L'= \set{T'}$ for $T,T' \in \mathcal{T}$. 
  Therefore, they need an additional assumption on the growth of the distance layers around each simplex~$T$ expressed in terms of constants~$\alpha$ and $\beta$. 
  Even with the grading estimates in \cite{GHS.2016} one has $\alpha=4$ and $\beta \geq 2$ for the 2D newest vertex bisection. 
  Since $0.318 \cdot \beta \alpha^{\frac 12} > 1$, their results do not include $W^{1,2}$-stability for 2D newest vertex bisection with $K=1$. 
\end{remark}

\section{The Operator~$C$ and the Metric~$\delta$}
\label{dec:design-C}
In this section we design a novel operator~$C\colon L^2(\Omega) \to \mathcal{L}_K^1(\tria)$ satisfying \ref{itm:self-adjoint}--\ref{itm:locality}.
We use a discrete partition of unity as weights for local weighted~$L^2$-projections rather than using the partition of unity and local unweighted $L^2$-projections as in \cite{BanYse14}. 
Remark~\ref{rem:comparison} discusses the differences in detail.
We verify \ref{itm:self-adjoint}, \ref{itm:distance} and \ref{itm:locality} in Subsection~\ref{subsec:DefOfC}, and property \ref{itm:ellipticity} is proved in the remaining Subsections \ref{sec:upper-eigenv-bound}--\ref{sec:CompLowerBound}. 

Recall, that $\vertices$ is the set of vertices
of~$\mathcal{T}$. 
Let $(\varphi_i)_{i \in \vertices}$ denote the Lagrange basis of~$\mathcal{L}^1_1(\mathcal{T})$, i.e., $\varphi_i(j) = \delta_{ij}$ for $i,j\in \vertices$. 
This basis serves as a discrete partition of unity $\sum_{i \in \vertices} \varphi_i = 1$ and the functions $\varphi_i$ are used as weights for our local weighted projections.

\subsection{Construction of $C$ and $\delta$}
\label{subsec:DefOfC}
With $(\varphi_i)_{i \in \vertices}$ we can decompose~$\mathcal{L}^1_K(\tria)$ as
\begin{align}
  \label{eq:space-decomp}
  \mathcal{L}_K^1(\tria) = \sum_{i\in \vertices} \varphi_i
  \mathcal{L}_{K-1}^1(\omega_i),
\end{align}
where $\omega_i$ denotes the patches around a vertex~$i \in \vertices$. 
Let us explain this decomposition. 
It is clear that the sum is contained in~$\mathcal{L}^1_K(\mathcal{T})$. 
To see the opposite inclusion let~$\psi$ be a Lagrange basis function of~$\mathcal{L}^1_K(\mathcal{T})$. 
Locally $\psi|_T$ contains as factor one of the barycentric coordinates~$\lambda_0, \dots, \lambda_d$ of $T \in \tria$, which corresponds to one of the~$\varphi_i$. 
We split off this factor to find~$\varphi_i$ and $v_{K-1} \in \mathcal{L}^1_{K-1}(\omega_i)$ such that $\psi = \varphi_i v_{K-1}$.

We set the weighted inner product $\skp{\cdot}{\cdot}_{\varphi_i} \coloneqq \skp{\varphi_i \cdot}{\cdot}$ on the local space $\mathcal{L}_{K-1}^1(\omega_i)$. 
The $L^2$-orthogonality of~$Q$ yields for all $v_{K-1}\in\mathcal{L}_{K-1}^1(\mathcal{T})$, $i\in
\vertices$ and $u \in L^2(\Omega)$
\begin{align}
  \label{eq:Qu-aux}
  \skp{Qu}{v_{K-1}}_{\varphi_i} =
  \langle  Qu,\varphi_i v_{K-1}\rangle = \langle 
  u, \varphi_i v_{K-1}\rangle =
  \skp{
  u}{v_{K-1}}_{\varphi_i}.
\end{align}
Thus, the difference $Qu-u$ is orthogonal to the space $\mathcal{L}^1_{K-1}(\mathcal{T})$ with respect to $\langle \cdot ,\cdot\rangle_{\varphi_i}$. 
Let $C_i$ denote the (local) orthogonal projection with respect to $\skp{\cdot}{\cdot}_{\varphi_i}$ mapping $L^2(\omega_i)$ onto the space $\mathcal{L}^1_{K-1}(\omega_i)$.
By restriction  of the domain, we can extend~$C_i$ to $C_i \colon L^2(\Omega) \to \mathcal{L}^1_{K-1}(\omega_i)$. 
Overall, for any $i \in \vertices$ we have that
\begin{align}
  \label{eq:Pjweighted}
  \skp{C_i u}{v_{K-1}}_{\varphi_i} = 
  \skp{u}{v_{K-1}}_{\varphi_i} \quad \text{ for all } v_{K-1} \in \mathcal{L}^1_{K-1}(\mathcal{T}). 
\end{align}
In particular, both $Q u - u$ and $C_i u -u$ are orthogonal to~$\mathcal{L}^1_{K-1}(\mathcal{T})$ with respect to the inner product $\skp{\cdot}{\cdot}_{\varphi_i}$.
This motivates us to approximate $Q$ by an operator $C \colon L^2(\Omega) \to \mathcal{L}^1_K(\mathcal{T})$ defined by the sum of the local operators $C_i$ as
\begin{align}\label{eq:def_C}
  C &\coloneqq \sum_{i\in \vertices} \varphi_i C_i.
\end{align}
Note that the structure of~$C$ is related to the space decomposition~\eqref{eq:space-decomp}. 

Let us introduce the distance function~$\delta$
on~$\mathcal{T}$ based on vertex neighborhoods.
\begin{definition}[{Distance} $\delta$]\label{def:Metric}
We call $T, T' \in \tria$ with $T \neq T'$ (nodal) neighbors, if they share a vertex. 
Then the induced (geodesic) distance function $\delta$ as in \ref{itm:distance} satisfies that $\delta(T,T') = 1$ if and only if $T \neq T'$ and $T \cap T' \neq \emptyset$.  
\end{definition}

\begin{theorem}[Properties of $C$ and $\delta$]\label{thm:propCd}
The operator $C$ and the distance $\delta$ defined in \eqref{eq:def_C} and in Definition~\ref{def:Metric}, respectively, satisfy \ref{itm:self-adjoint}--\ref{itm:locality}. 
More precisely,  
\begin{enumerate}[label=(\roman{*})]
\item \label{itm:verify-C1} 
$C$ is linear and self-adjoint with respect to the $L^2$-inner product~$\skp{\cdot}{\cdot}$; 
\item \label{itm:verify-C2} 
$C|_{\mathcal{L}^1_K(\mathcal{T})}$ is $L^2$-elliptic with condition number $\operatorname{cond}_2(C|_{\mathcal{L}^1_K(\mathcal{T})}) \leq \frac{2K+d}{K}$;
\item \label{itm:verify-C3} 
 $\delta$ is a geodesic distance function on~$\mathcal{T}$ induced by the vertex based notion of neighborhoods on $\tria$;
\item \label{itm:verify-C4} 
if $v \in L^2(\Omega)$ is supported on the collection~$L \subset \mathcal{T}$ of simplices, then the support of~$Cv$ is at most one layer of simplices larger. 
\end{enumerate}
Moreover, 
\begin{enumerate}[label=(\roman{*})]
\setcounter{enumi}{4}
\item \label{itm:Lk-1_id}
 $C$ is the identity on $\mathcal{L}^1_{K-1}(\tria)$. 
\end{enumerate} 
\end{theorem}
\begin{proof} 
Let us prove \ref{itm:verify-C1}, \ref{itm:verify-C3}, \ref{itm:verify-C4} and postpone the proof of \ref{itm:verify-C2} to Proposition~\ref{prop:lower-bound-comp} and \ref{itm:Lk-1_id} to Lemma~\ref{lem:K-1} below. 
\begin{enumerate}
\item[\ref{itm:verify-C1}] 
By definition the operator $C$ is linear and for any $u,w \in L^2(\Omega)$ one has that
  \begin{align*}
    \skp{C u}{w}
    &
    = \sum_{i\in \vertices} \skp{C_i u}{w}_{\varphi_i}
    = \sum_{i \in \vertices} \skp{u}{C_i w}_{\varphi_i}
      = \skp{u}{C w}.
  \end{align*}
\item[\ref{itm:verify-C3}] 
By definition the distance $\delta$ arises from the vertex based neighborhood notion and is a geodesic distance function as specified in \ref{itm:distance}. 
\item[\ref{itm:verify-C4}]  
By design of~$C$ and $\delta$ locality as in~\ref{itm:locality} is satisfied.  
\end{enumerate} 
\end{proof} 

The remaining property~\ref{itm:Lk-1_id} is proved in Subsection~\ref{sec:upper-eigenv-bound}, and~\ref{itm:verify-C2} is proved in the Subsections~\ref{sec:upper-eigenv-bound}--\ref{sec:CompLowerBound}. 

\begin{remark}[Operator~$C$ of Bank and Yserentant]
  \label{rem:comparison}
Let us explain the construction of the operator~$C$ used in~\cite{BanYse14} here denoted by~$\CBY$.
The local operators $\CBY_i$ are defined as the $\skp{\cdot}{\cdot}$-orthogonal projection to~$\varphi_i \mathcal{L}^1_{K-1}(\omega_i)$ for $i \in \vertices$. 
Note that if $\omega_i$ is an inner patch, then $\varphi_i \mathcal{L}^1_{K-1}(\omega_i)$ is the space of functions in $\mathcal{L}^1_K(\omega_i)$ vanishing on the boundary of~$\omega_i$. 
The operator $\CBY \colon L^2(\Omega)\to \mathcal{L}_K^1(\mathcal{T})$ reads
  \begin{align*}
    \CBY \coloneqq \sum_{i\in \vertices} \CBY_i.
  \end{align*}
Let us summarize the basic difference between the approaches. 
Bank and Yserentant use the unweighted $\skp{\cdot}{\cdot}$-projection to~$\varphi_i\mathcal{L}^1_{K-1}(\omega_i)$. 
In contrast, we use the weighted $\skp{\cdot}{\cdot}_{\varphi_i}$ projection to~$\mathcal{L}^1_{K-1}(\omega_i)$ and then multiply by~$\varphi_i$.

In order to improve the condition number of~$\CBY$ Bank and Yserentant orthogonally decompose the space~$\mathcal{L}^1_K(\mathcal{T})$ into $\mathcal{S}_0 + \mathcal{S}_1$, where $\mathcal{S}_1 \coloneqq \sum_{T \in \mathcal{T}}\mathcal{L}_{K,0}(T)$ with $\mathcal{L}_{K,0}(T) \coloneqq \set{v_K \in \mathcal{L}_K(T) \colon v_K|_{\partial T} =0}$. 
Let $Q_0$ and $Q_1$ denote the $\skp{\cdot}{\cdot}$-orthogonal projection to~$\mathcal{S}_0$ and $\mathcal{S}_1$, respectively. 
Since the operator~$Q_1$ acts locally, only the condition number of~$Q_0$ has to be investigated. 
Naturally, the condition number is smaller than the one of~$Q$. 
The projection and decomposition operators under consideration are preconditioned from both sides by $\identity -Q_1$.

For $H^1$-stability Bank and Yserentant need~$q< \frac12$ with $q$ as in~\eqref{eq:ConvSpeed2} subject to a grading $\gamma_h = 2$ (see Definition \ref{def:grading-cont} and \ref{def:MeshSizeFunction} for the definition of grading). 
 With the orthogonal space decomposition for $d=1$ one has that $q< \frac 12$ for all~$K \in \setN$, which motivates the decomposition. 
For $d=2$ this method increases the range of admissible~$K$ from $K=1,\dots, 8$ to $K=1, \dots, 12$.

Table \ref{tab:ComparBY} displays bounds on $q$ by Bank and Yserentant.
It indicates that in higher space dimensions the gain from the orthogonal decomposition is by far not as high as in the one-dimensional case. 
Moreover, the table compares the (numerically computed) values by Bank and Yserentant with the analytical estimates derived in Section \ref{dec:design-C}.
For $d\geq 2$ even the decomposition based ansatz leads to worse results than our analytical approach.
  
It is possible to improve our ansatz by the orthogonal decomposition used by Bank and Yserentant. 
However, numerical experiments indicate only a minor gain. 
	
	\begin{table}
	\centering
	\begin{TAB}[3pt]{|c|ccc|ccc|ccc|}{|c|c|cccccccccccccccc|}
		& &$d = 1$& & &$d=2$& & &$d=3$& \\
		$K$ & $q_\textup{BY}$ & $q_{\textup{pBY}}$ & $q_\textup{new}$ & $q_\textup{BY}$ & $q_{\textup{pBY}}$ & $q_\textup{new}$ &$q_\textup{BY}$ & $q_{\textup{pBY}}$ & $q_\textup{new}$\\
$	1	$&$	\highlight{	0.2679	}	$&$	\highlight{	0.2679	}	$&$	\highlight{	0.2679	}	$&$	\highlight{	0.3333	}	$&$	\highlight{	0.3333	}	$&$	\highlight{	0.3333	}	$&$	\highlight{	0.3820	}	$&$	\highlight{	0.3820	}	$&$	\highlight{	0.3820	}	$\\
$	2	$&$	\highlight{	0.2679	}	$&$	\highlight{	0.1716	}	$&$	\highlight{	0.2251	}	$&$	\highlight{	0.3564	}	$&$	\highlight{	0.3564	}	$&$	\highlight{	0.2679	}	$&$	\highlight{	0.4142	}	$&$	\highlight{	0.4142	}	$&$	\highlight{	0.3033	}	$\\
$	3	$&$	\highlight{	0.2419	}	$&$	\highlight{	0.1270	}	$&$	\highlight{	0.2087	}	$&$	\highlight{	0.3431	}	$&$	\highlight{	0.3150	}	$&$	\highlight{	0.2404	}	$&$	\highlight{	0.4930	}	$&$	\highlight{	0.4930	}	$&$	\highlight{	0.2679	}	$\\
$	4	$&$	\highlight{	0.2310	}	$&$	\highlight{	0.1010	}	$&$	\highlight{	0.2000	}	$&$	\highlight{	0.3375	}	$&$	\highlight{	0.2977	}	$&$	\highlight{	0.2251	}	$&$	\highlight{	0.4928	}	$&$	\highlight{	0.4829	}	$&$	\highlight{	0.2476	}	$\\
$	5	$&$	\highlight{	0.2289	}	$&$	\highlight{	0.0839	}	$&$	\highlight{	0.1946	}	$&$	\highlight{	0.3443	}	$&$	\highlight{	0.2809	}	$&$	\highlight{	0.2154	}	$&$		0.5112		$&$	\highlight{	0.4471	}	$&$	\highlight{	0.2344	}	$\\
$	6	$&$	\highlight{	0.2342	}	$&$	\highlight{	0.0718	}	$&$	\highlight{	0.1909	}	$&$	\highlight{	0.3529	}	$&$	\highlight{	0.2953	}	$&$	\highlight{	0.2087	}	$&$		0.5442		$&$	\highlight{	0.4727	}	$&$	\highlight{	0.2251	}	$\\
$	7	$&$	\highlight{	0.2491	}	$&$	\highlight{	0.0627	}	$&$	\highlight{	0.1883	}	$&$	\highlight{	0.3958	}	$&$	\highlight{	0.2922	}	$&$	\highlight{	0.2038	}	$&$		0.5878		$&$	\highlight{	0.4633	}	$&$	\highlight{	0.2183	}	$\\
$	8	$&$	\highlight{	0.2788	}	$&$	\highlight{	0.0557	}	$&$	\highlight{	0.1862	}	$&$	\highlight{	0.4350	}	$&$	\highlight{	0.3267	}	$&$	\highlight{	0.2000	}	$&$		0.6438		$&$		0.5064		$&$	\highlight{	0.2129	}	$\\
$	9	$&$	\highlight{	0.3284	}	$&$	\highlight{	0.0501	}	$&$	\highlight{	0.1847	}	$&$		0.5022		$&$	\highlight{	0.3342	}	$&$	\highlight{	0.1970	}	$&$		0.6978		$&$		0.5313		$&$	\highlight{	0.2087	}	$\\
$	10	$&$	\highlight{	0.3979	}	$&$	\highlight{	0.0455	}	$&$	\highlight{	0.1834	}	$&$		0.5664		$&$	\highlight{	0.3892	}	$&$	\highlight{	0.1946	}	$&$		0.7543		$&$		0.5966		$&$	\highlight{	0.2053	}	$\\
$	11	$&$	\highlight{	0.4805	}	$&$	\highlight{	0.0417	}	$&$	\highlight{	0.1823	}	$&$		0.6426		$&$	\highlight{	0.4116	}	$&$	\highlight{	0.1926	}	$&$		0.8021		$&$		0.6322		$&$	\highlight{	0.2024	}	$\\
$	12	$&$		0.5669		$&$	\highlight{	0.0385	}	$&$	\highlight{	0.1815	}	$&$		0.7074		$&$	\highlight{	0.4865	}	$&$	\highlight{	0.1909	}	$&$		0.8458		$&$		0.7044		$&$	\highlight{	0.2000	}	$\\
$	13	$&$		0.6493		$&$	\highlight{	0.0358	}	$&$	\highlight{	0.1807	}	$&$		0.7705		$&$		0.5255		$&$	\highlight{	0.1895	}	$&$		0.8803		$&$		0.7380		$&$	\highlight{	0.1979	}	$\\
$	14	$&$		0.7227		$&$	\highlight{	0.0334	}	$&$	\highlight{	0.1801	}	$&$		0.8201		$&$		0.6073		$&$	\highlight{	0.1883	}	$&$		0.9092		$&$		0.8001		$&$	\highlight{	0.1962	}	$\\
          \vdots & \vdots&\vdots&\vdots&\vdots&\vdots&\vdots&\vdots&\vdots&\vdots
          \\
$	\infty	$&$		\infty		$&$	\highlight{0}	$&$	\highlight{	0.1716	}	$&$		\infty		$&$		\infty		$&$	\highlight{	0.1716	}	$&$		\infty		$&$		\infty		$&$	\highlight{	0.1716	}	$
	\end{TAB}
	\caption{Values of $q$ in \eqref{eq:ConvSpeed2} for non-preconditioned Bank--Yserentant ($q_\textup{BY}$), preconditioned Bank--Yserentant ($q_\textup{pBY}$) and our (analytical) ansatz in Section \ref{dec:design-C} ($q_\textup{new}$) defined in \eqref{eq:qnew}. 
	We have recalculated the values of \cite{BanYse14} for a larger range here.} 
	\label{tab:ComparBY}
	\end{table}
\end{remark}

\subsection{Upper Bound on Eigenvalues}\label{sec:upper-eigenv-bound}
In this subsection we present further properties of~$C$ leading to an upper bound on its eigenvalues on $\mathcal{L}^1_K(\tria)$. 
\begin{lemma}[Behavior of~$C$ on~$\mathcal{L}^1_{K-1}(\mathcal{T})$]\ 
  \label{lem:K-1}
  \begin{enumerate}
  \item
    \label{itm:K-1a}
    The operator~$C$ is the identity on $\mathcal{L}^1_{K-1}(\mathcal{T})$;
  \item \label{itm:K-1b}
    The difference~$u -Cu$ is orthogonal to $\mathcal{L}^1_{K-1}(\mathcal{T})$
    for all $u \in L^2(\Omega)$.
  \end{enumerate}
\end{lemma}
\begin{proof}
  By~\eqref{eq:Pjweighted} it follows that~$C_i$ is the identity on~$\mathcal{L}^1_{K-1}(\omega_i)$. 
  Since $\sum_{i \in \vertices}\varphi_i=1$, this proves~\ref{itm:K-1a}. 
 The claim in~\ref{itm:K-1b} follows using~\ref{itm:K-1a} and the fact that $C$ is self-adjoint.
\end{proof}
Let us compute the largest eigenvalue $\lambda_\textup{max}(C|_{\mathcal{L}_K^1(\mathcal{T})})$ of $C$ restricted to $\mathcal{L}_K^1(\mathcal{T})$. 
\begin{lemma}[Upper bound]\label{lem:UpperBound}
  We have
  \begin{align*}
    \skp{C u}{u} &\leq \norm{u}_2^2 \qquad \text{ for all } u \in L^2(\Omega).
  \end{align*}
  The estimate is sharp and
  $\lambda_{\max}(C|_{\mathcal{L}^1_K(\mathcal{T})}) = 1 =
  \lambda_\textup{max}(C)$.
\end{lemma}
\begin{proof}
  Since $C_i$ is a projection with respect to $\langle\cdot,\cdot\rangle_{\varphi}$, for all $u \in L^2(\Omega)$ we find 
  \begin{align}\label{est:upper-bd}
    \skp{Cu}{u}
    &= \sum_{i\in \vertices} \skp{C_i u}{u}_{\varphi_i} = \sum_{i\in \vertices}
      \skp{C_i u}{C_i u}_{\varphi_i}
    \leq \sum_{i\in \vertices} \skp{u}{u}_{\varphi_i} = \norm{u}_2^2.
  \end{align}
  Since $C$ is self-adjoint, we have that
  \begin{align*}
    \lambda_\textup{max}(C) = \sup_{u\in L^2(\Omega)\setminus \lbrace
    0 \rbrace} \frac{\langle C u,u \rangle}{\lVert u\rVert_2^2} \leq 1.
  \end{align*} 
  By Lemma~\ref{lem:K-1} the operator~$C$ is the identity on~$\mathcal{L}_{K-1}^1(\mathcal{T})$. 
  Hence, the estimate is sharp.
\end{proof}

\subsection{Lower Bound on Eigenvalues using Decomposition Operators}
\label{sec:lower-eigenv-bound}\hfill

\noindent To show ellipticity \ref{itm:ellipticity} we have to bound the largest and the smallest eigenvalue of $C$. 
Since Lemma \ref{lem:UpperBound} shows that the largest eigenvalue equals one, it remains to derive a lower bound on the smallest eigenvalue $\lambda_\textup{min}(C|_{\mathcal{L}^1_K(\mathcal{T})})$. 
Similarly as Bank and Yserentant we estimate the smallest eigenvalue by means of a family of decomposition
operators. 

Suppose that we have operators $\frD_i \colon \mathcal{L}^1_K(\mathcal{T}) \to \mathcal{L}^1_{K-1}(\omega_i)$ decomposing functions in~$\mathcal{L}^1_K(\mathcal{T})$  according to the function space decomposition in~\eqref{eq:space-decomp}, i.e.,
\begin{align}\label{eq:decompProperty}
  v_K = \sum_{i\in \vertices} \varphi_i \frD_i v_K\qquad\text{for all }v_K\in \mathcal{L}_K^1(\mathcal{T}).
\end{align}
Note that strictly speaking the products $\varphi_i \frD_i \colon \mathcal{L}^1_K(\tria) \to \mathcal{L}^1_K(\tria)$ are the decomposition operators.
Such decomposition operators  are constructed in Section~\ref{sec:LocDecomp} below. 
Let $K_1$ be a constant such that for all $v_K \in \mathcal{L}^1_K(\mathcal{T})$ one has that
\begin{align}
  \label{eq:K1tilde}
  \sum_{i\in \vertices}\int_\Omega \varphi_i\, |\frD_i v_K|^2 \dx
  &\leq 
    K_1 \norm{v_K}_2^2.
\end{align}
\begin{lemma}[Lower bound]
  \label{lem:lower-bound}
  Suppose that there exist decomposition operators $\frD_i$ with \eqref{eq:decompProperty} and a constant $K_1$ as in \eqref{eq:K1tilde}. Then one has that
  \begin{align*}
    \frac{1}{K_1} \leq \lambda_{\min}(C|_{\mathcal{L}_K^1(\mathcal{T})}).
  \end{align*}
\end{lemma}
\begin{proof}
  Since $C_i$ is a projection with respect to $\langle \cdot,\cdot\rangle_{\varphi_i}$, all $u\in L^2(\Omega)$ satisfy that
  \begin{align}\label{eq:proofLowerBound}
    \skp{Cu}{u} &= \sum_{i\in \vertices} \skp{\varphi_i C_i u}{u}= \sum_{i\in \vertices} \skp{C_i u}{C_i u}_{\varphi_i}.
  \end{align}
  Let $v_K\in \mathcal{L}_K^1(\mathcal{T})$, then \eqref{eq:decompProperty}, \eqref{eq:proofLowerBound} and
  \eqref{eq:K1tilde} yield that
  \begin{align*}
    \norm{v_K}_2^2  &=
    \sum_{i\in \vertices} \skp{\varphi_i \frD_i v_K}{v_K}                                    
    = \sum_{i\in \vertices} \skp{\frD_i v_K}{v_K}_{\varphi_i}
      = \sum_{i\in \vertices} \skp{\frD_i v_K}{C_iv_K}_{\varphi_i}
    \\
     &\leq 
   \sum_{i \in \vertices} \left( \skp{\frD_i v_K}{\frD_i v_K}_{\varphi_i}^{1/2}  \skp{C_i v_K}{C_i v_K}_{\varphi_i}^{1/2} 
   \right)\\
     &\leq 
  \big( \sum_{i \in \vertices}  \skp{\frD_i v_K}{\frD_i v_K}_{\varphi_i}\big)^{1/2} 
    \big( \sum_{\ell \in \vertices}  \skp{C_\ell v_K}{C_\ell v_K}_{\varphi_\ell}\big)^{1/2}  \\ 
    &\leq \sqrt{K_1}\,
     \norm{v_K} _2 \sqrt{\skp{Cv_K}{v_K}}.
  \end{align*}
  This proves that $\norm{v_K}_2^2 \leq K_1\,  \skp{Cv_K}{v_K}$ for all $v_K \in \mathcal{L}^1_{K}(\tria)$. 
  Since $C$ is self-adjoint the claim follows.
\end{proof}

\subsection{Construction of the Decomposition Operators}\label{sec:LocDecomp}
In this subsection we construct the decomposition operators $\frD_i\colon \mathcal{L}^1_K(\mathcal{T})\to
\mathcal{L}^1_{K-1}(\omega_i)$ satisfying~\eqref{eq:decompProperty}. 
In combination with Lemma~\ref{lem:UpperBound} and \ref{lem:lower-bound} this allows us to verify condition~\ref{itm:ellipticity}.

We design the operators locally. 
That is, for fixed~$i \in \vertices$ we define $\frD_i$ locally on each~$T \in \omega_i$. 
Doing so we only have to ensure that the local objects are continuous on the boundaries of each~$T \in \omega_i$, such that~$\frD_i$ maps to~$\mathcal{L}^1_{K-1}(\omega_i)$.

An alternative view is to fix~$T \in \mathcal{T}$ and define~$\frD_i$ on~$T$ for those $i \in \vertices$ with $i \in T$. 
There are exactly $d+1$ such vertices that we relabel for simplicity by $j=0,\dots, d$. 
We denote these local operators by $\frD_j^T \colon \mathcal{L}_K(T) \to \mathcal{L}_{K-1}(T)$.
For the fixed $T \in \mathcal{T}$ let $\lambda_j$ denote barycentric coordinate that corresponds to the $j$-th vertex of $T$ for all $j=0,\dots,d$.
We define the monomial $\lambda^\sigma \coloneqq \lambda_0^{\sigma_0}\cdots\lambda_d^{\sigma_d}$ for all multi indices $\sigma = (\sigma_0,\dots,\sigma_d) \in \mathbb{N}_0^{d+1}$. 
Let us denote the Lagrange nodes of~$\mathcal{L}^1_K(T)$ by $x_\alpha$ with $\alpha \in \setN_0^{d+1}$ and $\abs{\alpha}=K$.
Let $e_j \in \mathbb{R}^{d+1}$ with $j=0,\dots,d$ denote the $j$-th canonical basis vector (we start counting from zero). 
With $\lambda_j(x_\alpha) = \alpha_j/K$, for $\abs{\alpha} = K$ we set
\begin{align}
  \label{eq:DefMononialDecomp2}
  \begin{aligned}
    \frD_j^T(\lambda^\alpha) & \coloneqq \frac{\alpha_j}{K}
    \lambda^{\alpha-e_j} =
    \begin{cases}
      \lambda_j(x_\alpha) \lambda^{\alpha-e_j} &\qquad \text{if
        $\alpha_j>0$},
      \\
      0 &\qquad \text{if $\alpha_j=0$,}
    \end{cases}
  \end{aligned}
\end{align}
with $0 \cdot \lambda^{\alpha - e_j} = 0$ by convention.
Then we obtain that
\begin{align}
  \label{eq:DefMononialDecomp}
  \lambda_j \frD_j^T(\lambda^\alpha) = \lambda_j(x_\alpha)
  \lambda^\alpha \qquad \text{for } j=0,\ldots, d\, \text{ and } \abs{\alpha}=K.
\end{align}
For any $i \in \vertices \cap T$ and $j \in \lbrace 0,\dots,d\rbrace$ being the local index of the vertex $i$ on $T$ we obtain the global operator mapping to $\mathcal{L}^0_{K-1}(\omega_i)$ by setting
\begin{align}\label{eq:def-Di-global}
(\frD_i v)|_T \coloneqq \frD_j^T(v|_T) \quad \text{ for all }  v \in \mathcal{L}^1_K(\tria).
\end{align}
As we use local coordinates to define~$\frD_j^T$, it is
universal and ``independent'' of~$T$.
Note that $\lbrace \lambda^\alpha \colon \alpha\in \mathbb{N}_0^{d+1}$ with $|\alpha| = K\rbrace$ forms a basis of the polynomial space $\mathcal{L}_K(T)$ on $T$ with maximal degree $K\in \mathbb{N}$. 
Hence, the identities in \eqref{eq:DefMononialDecomp2} define linear decomposition operators $\frD_j^T$ mapping from $\mathcal{L}_K(T)$ to $\mathcal{L}_{K-1}(T)$ for all $j=0,\dots,d$.
By~\eqref{eq:DefMononialDecomp} it follows immediately that
\begin{align*}
  \sum_{j=0}^d \lambda_j \frD_j^T(\lambda^\alpha) = \sum_{j=0}^d
   \lambda_{j}(x_\alpha) \lambda^\alpha = \lambda^\alpha.
\end{align*}
Therefore, the operators $\frD_j^T$ form a local decomposition in the sense that
\begin{align}
  \label{eq:local-decomp-DT}
  \sum_{j=0}^d \lambda_j \frD_j^Tv_K = v_K \qquad \text{for all $v_K
  \in \mathcal{L}_K(T)$}.
\end{align}

\begin{remark}[Integration formula for monomials]\label{rem:IntegrMonomials}
  For convenience let us recall for all multi indices $\sigma\in \mathbb{N}_0^{d+1}$ the formula
  \begin{align}\label{eq:integrationMonomials}
    \dashint_T \lambda^\sigma \dx =  \frac{1}{|T|} \int_T \lambda^\sigma \dx =
    \frac{\sigma!\,d!}{(|\sigma|+d)!} \coloneqq
    \frac{\sigma_0!\dots\sigma_d!\,d! }{(\sigma_0+\dots+\sigma_d+d)!}.
  \end{align}
\end{remark}
\begin{remark}[Motivation]
  \label{rem:motivation}
  The operators~$\frD_j^T$ are constructed by the following principles. 
  We need to ensure that the decomposition formula~\eqref{eq:decompProperty} holds locally. 
  Since the barycentric coordinates $\lambda_j$ are the
  restrictions of the $\varphi_i$, this is equivalent to
  \begin{align}
    \label{eq:motiv-constr} 
   \sum_{j=0}^d \lambda_j \frD_j^T(\lambda^\alpha) &= \lambda^\alpha \qquad\text{for all }\alpha\in \mathbb{N}_0^{d+1} \text{ with }\abs{\alpha}=K.
  \end{align}
  Let us choose the ansatz
  $\lambda_j \frD_j^T(\lambda^\alpha) \eqqcolon d_{j,\alpha}
  \lambda^\alpha$ for $d_{j,\alpha} \in \mathbb{R}$. 
  Since we aim for a small constant $K_1$ in \eqref{eq:K1tilde}, we choose the coefficients~$d_{j,\alpha} \in \mathbb{R}$ such that they minimize
  \begin{align}
    \label{eq:motiv-energy}
    \sum_{j=0}^d \int_T \lambda_j \abs{\frD_j^T (\lambda^\alpha)}^2\dx.
  \end{align}
  If $\alpha_j = 0$, the identity~\eqref{eq:motiv-constr} enforces~$d_{j,\alpha}=0$. 
  The remaining coefficients minimize
  \begin{align*}
    \sum_{j=0}^d d_{j,\alpha}^2 \int_T 
    \lambda^{2\alpha-e_j} \dx \quad \text{subject to the
    constraint} \quad
    \sum_{j=0}^d d_{j,\alpha} =1.
  \end{align*}
  This leads to the equivalent problem: 
  Seek coefficients $d_{j,\alpha}$ and a Lagrange multiplier $\mu\in \mathbb{R}$ with
    \begin{align}\label{eq:LagrangeMiulti1}
    2\, d_{j,\alpha} \frac{(2\alpha-e_j)!}{(2\abs{\alpha}+d-1)!} &= \mu\qquad\text{for all } j = 0,\dots,d\text{ with }\alpha_j >0.
  \end{align}
  By direct calculation we obtain that
  \begin{align*}
    2 &= 2\,\sum_{j=0}^d d_{j,\alpha} = \mu\, \sum_{j=0}^d
        \frac{(2\abs{\alpha}+d-1)!}{(2\alpha-e_j)!}
        = \mu\,\frac{(2\abs{\alpha}+d-1)!}{(2\alpha)!}
        \sum_{j=0}^d 2\alpha_j
    \\
    &= \mu\,\frac{2 \abs{\alpha}(2\abs{\alpha}+d-1)!}{(2\alpha)!}.
  \end{align*}
  Thus, with \eqref{eq:LagrangeMiulti1} we have that 
  \begin{alignat*}{2}
    \mu &=\frac{(2\alpha)!}{\abs{\alpha}(2\abs{\alpha}+d-1)!}
          \qquad \text{and} \qquad
    d_{j,\alpha} &= 
    \frac{\alpha_j}{\abs{\alpha}} = \frac{\alpha_j}{K}.
  \end{alignat*}
  This results in the identity~\eqref{eq:DefMononialDecomp}.
\end{remark}
So far we have defined the operator~$\frD_j^T$ only on monomials~$\lambda^\alpha$ with $\abs{\alpha}=K$. 
The next lemma shows how $\frD_j^T$ acts on monomials of lower order. 
\begin{lemma}[Lower order monomials]
\label{lem:Di-lower}
  For all $\gamma\in \mathbb{N}^{d+1}_0$ with $\abs{\gamma}\leq K$ we have
  \begin{align*}
    \lambda_j \frD_j^T (\lambda^\gamma) &= \frac{\gamma_j}{K} \lambda^\gamma
      +
      \frac{K-\abs{\gamma}}K \lambda_j \lambda^\gamma.
  \end{align*}
\end{lemma}
\begin{proof}
  Let $\gamma \in \mathbb{N}^{d+1}_0$ with $k\coloneqq \abs{\gamma} \leq K$ and let $j=0,\dots,d$. 
Since 
  \begin{align}
    \label{eq:1tom}
    1 &= \Big(\sum_{j=0}^d \lambda_j\Big)^m = \sum_{\abs{\sigma}=m}
        \binom{\abs{\sigma}}{\sigma} \lambda^\sigma 
  \end{align}
  is satisfied for all $m\in \mathbb{N}_0$, we find that
  \begin{align}\label{eq:proofLowerOrderMon}
    \lambda_j \frD_j^T (\lambda^\gamma)
    &= \lambda_j \frD_j^T \bigg(\lambda^\gamma \sum_{\abs{\sigma}=K-k}
      \binom{\abs{\sigma}}{\sigma} \lambda^\sigma \bigg)
      = \sum_{\abs{\sigma}=K-k}
      \binom{\abs{\sigma}}{\sigma} \lambda_j \frD_j^T (\lambda^{\gamma+\sigma}).
  \end{align}
The definition of~$\frD_j^T$ shows that
\begin{align*}
\lambda_j \frD_j^T(\lambda^{\gamma + \sigma})
= \lambda_j(x_{\gamma + \sigma}) \lambda^{\gamma + \sigma}  = \frac{\gamma_j + \sigma_j}{K} \lambda^{\gamma + \sigma}.
\end{align*}
Applying this in \eqref{eq:proofLowerOrderMon}, using \eqref{eq:1tom} and the substitution $\overline{\sigma} = \sigma - e_j$ yield that
  \begin{align*}
    \lambda_j \frD_j^T (\lambda^\gamma)
    &= \frac{\gamma_j}{K} \sum_{\abs{\sigma}=K-k}
      \binom{\abs{\sigma}}{\sigma}
      \lambda^{\gamma+\sigma}
      +
      \frac 1K \sum_{\substack{\abs{\sigma}=K-k\\\sigma_j>0}}
    \binom{\abs{\sigma}}{\sigma} \sigma_j
    \lambda^{\gamma+\sigma}
    \\
    &= \frac{\gamma_j}{K} \lambda^\gamma
      +
      \frac{K-k}K \sum_{\substack{\abs{\bar{\sigma}}=K-k-1}}
      \binom{\abs{\bar{\sigma}}}{\bar{\sigma}}
      \lambda^{\gamma+\bar{\sigma}} \lambda_j
    \\
    &= \frac{\gamma_j}{K} \lambda^\gamma
      +
      \frac{K-k}K \lambda_j \lambda^\gamma.
  \end{align*}  
\end{proof}
So far we know in particular that~$\frD_i \colon \mathcal{L}^1_K(\mathcal{T}) \to
\mathcal{L}^0_{K-1}(\mathcal{T})$. 
From the local decomposition~\eqref{eq:local-decomp-DT} it follows that the global decomposition operators~$\frD_i$ from \eqref{eq:def-Di-global} satisfy
\begin{align}
  \label{eq:global-decomp-D-aux}
  \sum_{i\in \vertices} \varphi_i \frD_i (v_K) = v_K \qquad\text{for all } v_K\in
  \mathcal{L}_K^1(\mathcal{T}). 
\end{align}
It remains to prove~$\frD_i \colon \mathcal{L}^1_K(\mathcal{T}) \to \mathcal{L}^1_{K-1}(\mathcal{T})$ which is based on the following lemma. 
\begin{lemma}[Traces]
  \label{lem:ContDi}
  Let $T \in \mathcal{T}$ and $v_K \in
  \mathcal{L}_K^1(\mathcal{T})$. Then for any subsimplex~$S$ of~$T$
  with $j \in \mathcal{N}\cap S$ the restriction $(\frD_j^T v_K)|_S$ only depends on
  $v_K|_S$.
\end{lemma}
\begin{proof}
  Let~$S$ be a subsimplex of~$T\in \tria$ and let $j \in \mathcal{N}\cap S$. 
  Since $\frD_j^T$ is linear it suffices to show for $v_K \in \mathcal{L}^1_K(\tria)$, that $v_K|_S=0$ implies that $(\frD_j^T v_K)|_S=0$. 
  So let us assume that $v_K|_S=0$. 
  The local representation of $v_K$
  reads $v_K = \sum_{|\alpha | = K} v_\alpha \lambda^\alpha$ with coefficients $v_\alpha \in \mathbb{R}$. 
  Let $x_\alpha$ denote the Lagrange nodes corresponding to $\lambda^\alpha \in \mathcal{L}_K(T)$. 
  Since $\lambda^\alpha|_S =0$ for $x_\alpha \not \in S$, we have that
  \begin{align*}
    0 = v_K|_S = \sum_{|\alpha | = K, x_\alpha \in S} v_\alpha \lambda^\alpha|_S.
  \end{align*}
  Because the functions $\lambda^\alpha|_S$ with $x_\alpha \in S$ form a
  basis of $\mathcal{L}_K(S)$ we obtain
  $v_\alpha=0$ for all $x_\alpha \in S$.
  Since for each term in the sum $v_\alpha$ or $\lambda^\alpha|_S$ equals zero, we obtain 
  \begin{align*}
    (\lambda_j \frD_j^T v_K)|_S = \sum_{|\alpha | = K}v_\alpha
    \frac{\alpha_j}{K} \lambda^{\alpha}|_S=0.
  \end{align*}
  With $j \in S$,
  we have $\lambda_j>0$ almost everywhere on~$S$ and thus we can divide
  by~$\lambda_j$. 
  We obtain $ (\frD_j^T v_K)|_S = 0$ as desired.
\end{proof}
In summary we obtain that the global decomposition operators defined in \eqref{eq:def-Di-global} satisfy the desired properties as assumed in Section~\ref{sec:lower-eigenv-bound}.
\begin{proposition}[Decomposition]\label{prop:decoomp}
  The operators $\frD_i$ are linear and map from
  $\mathcal{L}^1_K(\mathcal{T})$ to
  $\mathcal{L}_{K-1}^1(\omega_i)$. Moreover, the operators $\varphi_i \frD_i$ map 
  $\mathcal{L}^1_K(\mathcal{T}) \to \mathcal{L}^1_K(\tria)$ and decompose $\mathcal{L}_K^1(\mathcal{T})$-functions in the sense that 
    \begin{align*}
    v_K = \sum_{i\in \vertices} \varphi_i \frD_i v_K\qquad\text{for all }v_K\in \mathcal{L}_K^1(\mathcal{T}).
  \end{align*} 
\end{proposition}
\begin{proof}
  As a consequence of Lemma~\ref{lem:ContDi} we see that $\frD_iv_K$ is continuous on~$\omega_i$ and are therefore elements of~$\mathcal{L}^1_{K-1}(\omega_i)$. 
  Hence, $\varphi_i \frD_i v_K \in \mathcal{L}^1_K(\mathcal{T})$. 
  We have already shown the decomposition formula
  in~\eqref{eq:global-decomp-D-aux}. 
\end{proof}
\begin{remark}[Decomposition of Bank and Yserentant]	  
  Our decomposition relies on the function space decomposition $\mathcal{L}_K^1(\tria) = \sum_{i\in \vertices} \varphi_i \mathcal{L}_{K-1}^1(\omega_i)$, see~\eqref{eq:space-decomp}. 
  In particular, we treat the functions~$\varphi_i$ as weights and use decomposition operators~$\varphi_i \frD_i$ with the weight outside.
  Bank and Yserentant instead use the decomposition
  \begin{align*}
    v_K &= \sum_{i \in \vertices} \Pi_{\mathcal{L}^1_K}(\varphi_i v_K),
  \end{align*}
  where $\Pi_{\mathcal{L}^1_K}$ is the Lagrange interpolation operator mapping to~$\mathcal{L}^1_K(\mathcal{T})$.
  Similarly as our operators~$\frD_i$, the operators $\frac{1}{\varphi_i} \Pi_{\mathcal{L}^1_K}(\varphi_i\, \cdot)$ map $\mathcal{L}^1_K(\mathcal{T})$ to $\mathcal{L}^1_{K-1}(\omega_i)$. 

  In principle we can use the decomposition of Bank and Yserentant. 
  However, our decomposition is adapted to the space decomposition~\eqref{eq:space-decomp} and the locally weighted quantity $\sum_{i \in \vertices} \skp{\frD_i v_K}{\frD_i v_K}_{\varphi_i}$ as in~\eqref{eq:K1tilde}. 
  This results in a  smaller constant $K_1$. 
\end{remark}

\begin{remark}[Zero traces]
  \label{rem:D-dirichlet}
  Lemma~\ref{lem:ContDi} implies that the operators $\varphi_i\frD_i$ preserve zeros traces. 
  Indeed, let $v_K \in \mathcal{L}^1_K(\mathcal{T})$ with $v_K|_S=0$ on some subsimplex $S \subset \partial \Omega$. For all $i\in \mathcal{N}$ with $i \notin S$, we have $(\varphi_i \frD_i v_K)|_S=0$. 
  If $i \in S$, then Lemma~\ref{lem:ContDi} shows $(\varphi_i \frD_i v_K)|_S=0$.
  In particular, if $v_K \in \mathcal{L}^1_K(\mathcal{T})$ satisfies $ v_K|_{\Gamma}=0$ for some subset $\Gamma \subset \partial \Omega$ (compatible with the triangulation), then $(\varphi_i \frD_i
  v_K)|_{\Gamma}=0$. 
\end{remark}

\subsection{Computation of the Lower Bound}\label{sec:CompLowerBound}

In the course of this section we shall compute the value $K_1$ in \eqref{eq:K1tilde} and thus obtain a lower bound on $\lambda_{\min}(C|_{\mathcal{L}^1_K(\tria)})$. 

\begin{proposition}[Values of Lower Bound]\label{prop:lower-bound-comp}
For the operators $\frD_i$ we obtain 
\begin{align}
  \label{eq:K1-expl}
  \sum_{i\in \vertices}\int_\Omega \varphi_i\, |\frD_i v_K|^2 \dx
  &\leq 
    \frac{2K+d}{K} \norm{v_K}_2^2 \quad \text{ for all } v_K \in \mathcal{L}^1_K(\tria),
\end{align}
i.e., \eqref{eq:K1tilde} is satisfied with $K_1 = (2K+d)/K$. 
Furthermore, $C$ satisfies \ref{itm:ellipticity} with 
  \begin{align*}
    \operatorname{cond}_2(C|_{\mathcal{L}_K^1(\mathcal{T})}) \leq \frac{2K+d}{K}.
  \end{align*}
\end{proposition}
\begin{proof}
For now let us assume that \eqref{eq:K1tilde} is satisfied with $K_1= (2K+d)/K$, which we show in the remaining section, see Proposition~\ref{prop:eigenvaluesS} below. 
By Lemma~\ref{lem:lower-bound} we have $\lambda_{\min}(C|_{\mathcal{L}^1_K(\tria)}) \geq 1/K_1$, and by Lemma~\ref{lem:UpperBound} it follows that 
\begin{align*}
\operatorname{cond}_2(C|_{\mathcal{L}_K^1(\mathcal{T})}) = \frac{\lambda_{\max}(C|_{\mathcal{L}^1_K(\tria)})}{\lambda_{\min}(C|_{\mathcal{L}^1_K(\tria)})} \leq K_1 = \frac{2K+d}{K}. 
\end{align*}   
\end{proof}

For the proof of Proposition~\ref{prop:lower-bound-comp} it remains to show that \eqref{eq:K1tilde} is satisfied with $K_1 = (2K+d)/K$. 
For this it suffices to show that locally we have
\begin{align}\label{eq:K1loc}
  \sum_{j=0}^d \int_T \lambda_j\,|\frD_j^T v_K|^2\dx \leq K_1 \norm{ v_K }_{2,T}^2 \qquad\text{for all }v_K\in \mathcal{L}_K(T) \text{ and all } T\in \mathcal{T}.
\end{align}
Scaling shows that the constant~$K_1$ is independent of the simplex~$T\in \mathcal{T}$.
In the remaining subsection we explicitly compute the constant $K_1$. 

The mapping $(v,w) \mapsto \sum_{j=0}^d \langle \lambda_j \frD_j^T v, \frD_j^T w\rangle_T$ is a symmetric bilinear form on $\mathcal{L}_K(T) \times \mathcal{L}_K(T)$. 
Hence, there exists a self-adjoint and positive semi-definite operator $\frS\colon \mathcal{L}_K(T) \to \mathcal{L}_K(T)$ with
\begin{align*}
  \skp{\frS v_K}{w_K}_T &= \sum_{j=0}^d \langle \lambda_j \frD_j^T v_K,\frD_j^T
                      w_K\rangle_T\qquad\text{for all }v_K,w_K\in \mathcal{L}_K(T).
\end{align*}
Let us show that $\frS$ is also positive definite. 
Indeed, if $\skp{\frS v_K}{v_K}_T=0$ for some $v_K\in \mathcal{L}_K(T)$, then $\sum_{j=0}^d \langle \lambda_j \frD_j^T v_K,\frD_j^T v_K\rangle_T=0$, which implies $\frD_j^T v_K=0$ for all $j = 0,\dots,d $. 
With the decomposition $v_K = \sum_{j=0}^d \lambda_j \frD_j^T v_K$, we conclude that $v_K=0$. 
This proves that $\frS$ is positive definite.
We can rewrite~\eqref{eq:K1loc} as
\begin{align}
  \label{eq:S_K1}
  \skp{\frS v_K}{v_K} \leq K_1 \norm{ v_K }_{2,T}^2 \qquad\text{for all }v_K\in \mathcal{L}_K(T).
\end{align}
Thus, the smallest constant~$K_1$ in \eqref{eq:K1loc} is the largest
eigenvalue of~$\frS$.

Before characterizing the eigenspaces and eigenvalues of $\frS$ we investigate its action on monomials.
Recall the convention $0\cdot \lambda^\sigma = 0$ for all multi indices $\sigma\in \mathbb{Z}^{d+1}$.
\begin{lemma}[Operator $\frS$]
  For all multi indices $\sigma\in \mathbb{N}^{d+1}_0$ with $|\sigma| \leq K$ the operator $\frS$ satisfies
  \begin{align}\label{eq:S_lambda}
    \frS \lambda^\sigma = \frac{K^2 + \abs{\sigma}(\abs{\sigma}+d)}{K^2}
    \lambda^\sigma
    + \sum_{j=0}^d
    \frac{\sigma^2_j}{K^2} \lambda^{\sigma-e_j}.
  \end{align}
\end{lemma}
\begin{proof}
  Let $\alpha,\sigma\in \mathbb{N}_0^{d+1}$ be multi indices with $\abs{\sigma} \leq K = \abs{\alpha}$. 
  Lemma~\ref{lem:Di-lower} implies
  \begin{align*}
    \lefteqn{\skp{\frS \lambda^\sigma}{\lambda^\alpha}_T
    =
    \sum_{j=0}^d \langle \lambda_j \frD_j^T
    (\lambda^\sigma),\frD_j^T (\lambda^\alpha)\rangle_T} \quad
    &
    \\
    &=
      \sum_{j=0}^d \biggskp{ \frac{\sigma_j}{K}
      \lambda^{\sigma-e_j} + \frac{K-\abs{\sigma}}{K}
      \lambda^\sigma}{ \frac{\alpha_j}{K} \lambda^\alpha} _T
    \\
    &=
      \sum_{j=0}^d \int_T \left( \frac{\sigma_j \alpha_j}{K^2} \lambda^{\sigma+\alpha-e_j} +
      \frac{K-\abs{\sigma}}{K}     \frac{\alpha_j}{K}\lambda^{\sigma+\alpha} \right)
      \dx 
    \\
    &= d!\, |T|
      \sum_{j=0}^d \left(\frac{\sigma_j(\sigma_j+\alpha_j) - \sigma_j^2}{K^2}
      \frac{(\sigma+\alpha-e_j)!}{(\abs{\sigma+\alpha}+d-1)!} +
      \frac{K-\abs{\sigma}}{K}  \frac{\alpha_j}{K}
      \frac{(\sigma+\alpha)! }{(\abs{\sigma+\alpha}+d)!}\right)  
    \\
    &=
      d!\, |T| \sum_{j=0}^d \left( \frac{\sigma_j(\abs{\sigma+\alpha}+d)}{K^2}
      \frac{(\sigma+\alpha)!}{(\abs{\sigma+\alpha}+d)!} - \frac{\sigma_j^2}{K^2}
      \frac{(\sigma+\alpha-e_j)!}{(\abs{\sigma+\alpha}+d-1)!} \right)\\
    &\qquad \quad + d!\, |T|\, \frac{K-\abs{\sigma}}{K}
      \frac{(\sigma+\alpha)! }{(\abs{\sigma+\alpha}+d)!}      
    \\
    &=
      \biggskp{\frac{|\sigma| (|\sigma|+K+d)}{K^2} \lambda^\sigma - \sum_{j=0}^d \frac{\sigma_j^2}{K^2} \lambda^{\sigma-e_j} + \frac{K-\abs{\sigma}}{K} \lambda^\sigma }{\lambda^\alpha}_{\!\!T}
    \\
    &=
      \biggskp{\frac{K^2 + \abs{\sigma}(\abs{\sigma}+d)}{K^2}
      \lambda^\sigma
      + \sum_{j=0}^d
      \frac{\sigma^2_j}{K^2} \lambda^{\sigma-e_j}}{\lambda^\alpha}_{\!\!T}.
  \end{align*}
  Since $\mathcal{L}_K(T) = \textup{span}\lbrace \lambda^\alpha \colon \alpha\in \mathbb{N}^{d+1}_0$ with $\alpha = |K|\rbrace$, the calculation proves \eqref{eq:S_lambda}.
\end{proof}

The following proposition shows that the eigenspaces of $\frS$ are the orthogonal spaces $\mathcal{Z}_k$, for $k=0,\dots,K$, defined by
\begin{align}\label{eq:defZk}
  \mathcal{Z}_k \coloneqq \lbrace v_k \in \mathcal{L}_k(T)\colon \langle v_k,w_{k-1}\rangle_T=0\text{ for all }w_{k-1}\in \mathcal{L}_{k-1}(T)\rbrace.
\end{align}
For the computation we use ideas from \cite[Theorem~11]{Der85}.

\begin{proposition}[Eigenvalues]\label{prop:eigenvaluesS}
  The eigenspaces of~$\frS\colon\mathcal{L}_K(T)\to \mathcal{L}_K(T)$ are the spaces $\mathcal{Z}_k$ and the corresponding eigenvalues are given by
  \begin{align*}
  \mu_k \coloneqq \frac{K^2+k(k+d)}{K^2}\qquad\text{for }k=0,\dots,K.
  \end{align*}
  In particular, the inequality in~\eqref{eq:K1tilde} holds with $K_1 = (2K+d)/K$.
\end{proposition}
\begin{proof}
   \textit{Step 1 (Show that
    $\frS \mathcal{Z}_k \subset \mathcal{Z}_k$).} 
    The identity in \eqref{eq:S_lambda} verifies
  \begin{align}\label{eq:inclusion}
    \frS \mathcal{L}_k(T) \subset \mathcal{L}_k(T)\qquad\text{for all }k=0,\dots,K.
  \end{align}
  This and the orthogonality of~$\mathcal{Z}_k$ on $\mathcal{L}_{k-1}(\mathcal{T})$ shows that
  \begin{align}\label{eq:identSinZk}
    \langle \frS z_k,w_{k-1}\rangle_T = \langle z_k, \frS
    w_{k-1}\rangle_T = 0 \qquad \text{for all } z_k\in \mathcal{Z}_k, \, w_{k-1}\in\mathcal{L}_{k-1}(T).
  \end{align}
  Combining \eqref{eq:inclusion} and  \eqref{eq:identSinZk} implies that $\frS \mathcal{Z}_k \subset \mathcal{Z}_k$ for all $k=0,\dots,K$.

  \textit{Step 2 ($\mathcal{Z}_k$ is spanned by eigenfunctions).}
  Since $\frS|_{\mathcal{Z}_K}\colon\mathcal{Z}_k \to \mathcal{Z}_k$, $k=0,\dots,K$, is a self-adjoint and positive definite operator, there exist linearly independent eigenfunctions $z_{k,1},\dots,z_{k,\dim \mathcal{Z}_k} \in \mathcal{Z}_k$ with eigenvalues $\mu_{k,1},\dots,\mu_{k,\dim \mathcal{Z}_k} > 0$ in the sense that any $j = 0,\dots,\dim \mathcal{Z}_k$ we have that
  \begin{align}\label{eq:EigenspacesS}
    \langle \frS z_{k,j},w_k\rangle_T = \mu_{k,j} \, \langle z_{k,j},w_k\rangle_T\qquad\text{for all }w_k \in \mathcal{Z}_k.
  \end{align}
  The orthogonality of $\frS z_{k,j}\in \mathcal{Z}_k$ on $\mathcal{Z}_\ell$ with $\ell\in \mathbb{N}\setminus \lbrace k \rbrace$ extends the identity in \eqref{eq:EigenspacesS} to all functions in $\mathcal{L}_K(T)$. 
  This means that one has that
  \begin{align}\label{eq:EVPzk}
    \langle \frS z_{k,j},w_K\rangle_T = \mu_{k,j} \, \langle z_{k,j},w_K\rangle_T \quad \text{ for all } w_K \in \mathcal{L}_K(T).
  \end{align}
  Hence, the function $z_{k,j}$ is also an eigenfunction of $\frS \colon \mathcal{L}_K(T) \to \mathcal{L}_K(T)$.
  
  \textit{Step 3 (Computation of the eigenvalues).} 
For $k = 0,\dots,K$ let $z_k \in \mathcal{Z}_k \setminus \lbrace 0 \rbrace$ be an eigenfunction of $\frS$ with eigenvalue $\mu_k$, that is, $(z_k,\mu_k) \in \mathcal{Z}_k\times \mathbb{R}$ satisfies \eqref{eq:EVPzk}. 
There exist coefficients $z_\sigma\in \mathbb{R}$, for $\abs{\sigma} = k$, and a function $v_{k-1}\in \mathcal{L}_{k-1}(T)$ with
  \begin{align*}
    z_k = \sum_{|\sigma| = k} z_\sigma \lambda^\sigma + v_{k-1}.
  \end{align*}
  The identity in \eqref{eq:S_lambda} and the $L^2$-orthogonality of $z_k$ and $\frS z_k$ on $\mathcal{L}_{k-1}(T)$ imply 
 \begin{align*}
 \mu_k\, \langle z_k,z_k\rangle_T 
 &= \langle \frS z_k, z_k\rangle_T  
 =  
      \frac{K^2 + \abs{\sigma}(\abs{\sigma}+d)}{K^2}
      \sum_{|\sigma| = k} \langle z_\sigma \lambda^\sigma
      ,z_k\rangle_T \\
 &=    \frac{K^2 + \abs{\sigma}(\abs{\sigma}+d)}{K^2} \langle
    z_k,z_k\rangle_T.
  \end{align*}  
  This shows that any eigenfunction of $\frS$ in $\mathcal{Z}_k$ (Step 2) has eigenvalue 
  \begin{align*}
    \mu_k &= \frac{K^2 + \abs{\sigma}(\abs{\sigma}+d)}{K^2} = \frac{K^2 + k(k+d)}{K^2}.
  \end{align*}
  For $k=K$ we obtain the largest eigenvalue~$\frac{2K+d}{K}$. 
  This proves~\eqref{eq:S_K1}, \eqref{eq:K1loc} and in particular \eqref{eq:K1tilde} with~$K_1= \frac{2K+d}{K}$.
\end{proof}
\begin{remark}[Sharp estimate]
  The estimate
  $\operatorname{cond}_2(C|_{\mathcal{L}_K^1(\mathcal{T})})\leq
  \frac{2K+d}{K}$ in Proposition~\ref{prop:lower-bound-comp} is
  essentially sharp.

  First let us show this for a triangulation~$\mathcal{T}$ that consists of one single $d$-simplex~$T$. 
  To avoid confusion we denote the operators~$C$
  and $C_i$ in this situation by~$\Cloc$ and $\Cloc_i$,
  respectively.  
  By Lemma~\ref{lem:UpperBound} we know that
  $\lambda_{\max} (\Cloc|_{\mathcal{L}_K^1(\mathcal{T})})=1$, and Theorem~\ref{thm:propCd} shows that $1$ is an eigenfunction of $\Cloc$ with eigenspace $\mathcal{L}_{K-1}(T)$. 
  It remains to show that
  $\lambda_{\min} (\Cloc|_{\mathcal{L}_K(T)})=\frac{K}{2K+d}$. 
  For this purpose,
  we prove that the non-zero elements of~$\mathcal{Z}_K$ are
  eigenfunctions of~$\Cloc$ with eigenvalue $\frac{K}{2K+d}$.

  By \cite[Sec.\ 5.3]{DunXu14} for the space~$\mathcal{L}_K(T)$ there are $\skp{\cdot}{\cdot}$-biorthogonal bases $\mathcal{B}_\alpha$ and
  $\mathcal{M}_\alpha$ indexed over 
  $\alpha \in \setN_0^{d+1}$ with
  $\alpha_0=0$ and $\abs{\alpha}\leq K$ such that
  $\mathcal{B}_\alpha = \partial_\lambda^\alpha(\lambda^\alpha
  \lambda_0^{\abs{\alpha}})$, 
  $\mathcal{M}_\alpha-\lambda^\alpha \in \mathcal{L}_{K-1}(T)$, and
  $\skp{\mathcal{B}_\alpha}{\mathcal{M}_\beta}= 0$ if
  $\alpha\neq \beta$. 
  Moreover, for $i=0,\dots,d$ we find 
  $\skp{\cdot}{\cdot}_{\lambda_i}$-orthogonal bases $\mathcal{B}_\sigma^i$ and $\mathcal{M}_\sigma^i$ of $\mathcal{L}_{K-1}(T)$ indexed over $\sigma \in \setN_0^{d+1}$ with $\sigma_i=0$ and $\abs{\sigma}\leq K-1$ such that
  $\mathcal{B}^i_\sigma = \lambda_i^{-1} \partial_\lambda^\sigma(\lambda^{\sigma+e_i}
  \lambda_0^{\abs{\sigma}})$ and
  $\mathcal{M}^i_\sigma-\lambda^\sigma \in \mathcal{L}_{K-2}(T)$ and
  $\skp{\mathcal{B}^i_\sigma}{\mathcal{M}^i_\tau}_{\lambda_i}= 0$ if
  $\sigma\neq \tau$.
  This allows us to represent~$\Cloc_i$ for $v \in \mathcal{L}_K(T)$ as
  \begin{align}
    \label{eq:repr_Cloc_i}
    \Cloc_i v
    = \sum_{|\sigma|\leq K-1,\sigma_0 = 0} \frac{\skp{v}{\lambda_i \mathcal{M}^i_\sigma}}{\skp{ \mathcal{B}_\sigma^i}{\lambda_i \mathcal{M}^i_\sigma}}\mathcal{B}_\sigma^i.
  \end{align}
  Suppose that $\bar{\alpha} \in \setN_0^{d+1}$ with $\bar{\alpha}_0=0$ and $\abs{\bar{\alpha}}=K$. 
  The polynomials $\mathcal{B}_{\bar{\alpha}}$ for such~$\bar{\alpha}$ form a basis of~$\mathcal{Z}_K$.
  By~\eqref{eq:repr_Cloc_i} and the biorthogonality it follows that
  \begin{align}
    \label{eq:repr_Cloc_i2}
    \Cloc_i \mathcal{B}_{\bar{\alpha}}
    = \sum_{|\sigma|\leq K-1,\sigma_0 = 0} \frac{
       \skp{\mathcal{B}_{\bar{\alpha}}}{\lambda^{\bar{\alpha}+e_i}}
    }{
    \skp{ \mathcal{B}_\sigma^i}{\lambda^{\bar{\alpha}+e_i}}
    }\mathcal{B}_\sigma^i.
  \end{align}
  Let us define $\delta_{\alpha_i>0}$ as $1$ if $\alpha_i>0$ and as~$0$ if not.
  From~\eqref{eq:repr_Cloc_i2},
  the definition of $\mathcal{B}_{\bar{\alpha}}$, and a short calculation using integration by parts it follows that
  \begin{align}
    \label{eq:repr-aux1}
    \begin{aligned}
    \Cloc_i \mathcal{B}_{\bar{\alpha}}
    &
      = \delta_{\alpha_i>0} \, \frac{\skp{\mathcal{B}_{\bar{\alpha}} }{\lambda^\alpha}}{\skp{ \mathcal{B}_{\bar{\alpha}}^i}{ \lambda^{\bar{\alpha}}}}\mathcal{B}_{\bar{\alpha}}^i
      = - \delta_{\alpha_i>0}\; \alpha_i \frac{K}{2K+d} \mathcal{B}_{\bar{\alpha}-e_i}^i \qquad \text{for $i\neq 0$},
    \\
    \Cloc_0 \mathcal{B}_{\bar{\alpha}}
    &
      = - \sum_{j=1}^d \delta_{\alpha_j>0} \frac{\skp{\mathcal{B}_{\bar{\alpha}} }{\lambda^{\bar{\alpha}}}}{\skp{ \mathcal{B}_{\bar{\alpha}}^0}{ \lambda^{\bar{\alpha} +e_0-e_j}}}\mathcal{B}_{\bar{\alpha}-e_j}^0
      = \sum_{j=1}^d \delta_{\alpha_j>0} \frac{\alpha_j^2}{2K+d} \mathcal{B}_{\bar{\alpha}-e_j}^0.
    \end{aligned}
  \end{align}
  Now,~\eqref{eq:repr-aux1} and
  $1 = \sum_{j=1}^d \alpha_j/K$ 
  yield that
  \begin{align*}
    \Cloc \mathcal{B}_{\bar{\alpha}}
    &= \sum_{i=0}^d \lambda_i \Cloc_i \mathcal{B}_{\bar{\alpha}} = \frac{K}{2K+d} \sum_{j=1}^d\frac{\bar{\alpha}_j}{K}
      \partial^{\bar{\alpha} - e_j}(\bar{\alpha}_j \lambda^{\bar{\alpha}-e_j} - K \lambda^{\bar{\alpha}} \lambda_0^{|\bar{\alpha}|-1})
    \\
    & = \frac{K}{2K+d} \sum_{j=1}^d \frac{\bar{\alpha}_j}{K} \partial^{\bar{\alpha}}(\lambda^{\bar{\alpha}} \lambda_0^{|\bar{\alpha}|}) = \frac{K}{2K+d} \mathcal{B}_{\bar{\alpha}}.
  \end{align*}
  This proves that~$\mathcal{B}_{\bar{\alpha}}$ is an eigenfunction
  of~$\Cloc$ with eigenvalue~$\frac{K}{2K+d}$. Thus, every non-zero
  element of~$\mathcal{Z}_K$ is also an eigenfunction of~$\Cloc$ with
  eigenvalue~$\frac{K}{2K+d}$.
  Hence, sharpness of 
  $\operatorname{cond}_2(C|_{\mathcal{L}_K^1(\mathcal{T})})=
  \frac{2K+d}{K}$ is proved in the case of a triangulation consisting of a single $d$-simplex.

  The estimate is also sharp for general triangulations if $K \geq 2$.
  In this case we use the local eigenfunction $\mathcal{B}_{(K-1)e_1+e_2}$ of~$\Cloc$ and symmetrize it
  over~$(\lambda_0,\dots, \lambda_d)$ to obtain a fully symmetric
  eigenfunction of~$\Cloc$ with eigenvalue~$\frac{K}{2K+d}$. 
  For arbitrary triangulations the local symmetric functions can be patched together to obtain a global continuous function~$z \in \mathcal{L}^1_K(\mathcal{T})$, which is an eigenfunction of~$C$ with eigenvalue~$\frac{K}{2K+d}$.

  For $K=1$ we cannot symmetrize the local eigenfunction and thus,
  consider slightly less general triangulations. 
  Assume that the
  triangulation is $d+1$-colorable, i.e., each vertex of a simplex is assigned with a different number in the set $\set{0,\dots,d}$. 
  For example the uniform triangulation
  of a square and all its refinements using the bisection algorithm by Maubach and Traxler satisfy this condition. 
  Then, again the local eigenfunctions  $\mathcal{B}_{e_1}$
  of~$\Cloc$ can be patched together to a global eigenfunction~$z \in \mathcal{L}^1_1(\mathcal{T})$ of~$C$ with eigenvalue~$\frac{1}{2+d}$.
\end{remark}

\subsection{Zero Boundary Values}
\label{sec:dirichl-bound-valu}

In this subsection we consider the $L^2$-projection onto Lagrange spaces with zero boundary traces. 
Again let $\Omega$ be a bounded, polyhedral domain in~$\Rd$ with $d\in \mathbb{N}$. 
Moreover, let~$\Gamma \subset \partial \Omega$ denote the part of the boundary with zero boundary values. 
We assume that~$\Gamma$ is resolved by the triangulation and we define the space
\begin{align*}
  \mathcal{L}_{K,\Gamma}^1(\mathcal{T}) \coloneqq \set{ v\in
  \mathcal{L}_K^1(\mathcal{T}) \colon v|_{\Gamma}  = 0}. 
\end{align*}
By $Q_{\Gamma}$ we denote the $L^2$-projection to~$\mathcal{L}^1_{K,\Gamma}(\mathcal{T})$ with respect to~$\skp{\cdot}{\cdot}$.

As in~\eqref{eq:space-decomp} we decompose $\mathcal{L}_{K,\Gamma}^1(\tria)$ using the Lagrange basis~$\varphi_i$ of $\mathcal{L}^1_1(\mathcal{T})$. 
With $\mathcal{L}_{K-1,\Gamma}^1(\omega_i) \coloneqq \set{ v_{K-1} \in  \mathcal{L}_{K-1}^1(\omega_i)\colon \varphi_i v_{K-1} \in \mathcal{L}_{K,\Gamma}^1(\tria)}$
the decomposition reads 
\begin{align}
  \label{eq:space-decomp-D}
  \mathcal{L}_{K,\Gamma}^1(\tria) = \sum_{i\in \vertices} \varphi_i
  \mathcal{L}_{K-1,\Gamma}^1(\omega_i).
\end{align}
Let $C_{i,\Gamma}$ denote the (local) orthogonal projection with respect to the inner product
$\skp{\cdot}{\cdot}_{\varphi_i}$ mapping $L^2(\omega_i)$ onto the
space $\mathcal{L}^1_{K-1,\Gamma}(\omega_i)$. 
As before, we can extend~$C_{i,\Gamma}$ to $C_{i,\Gamma} \colon L^2(\Omega) \to \mathcal{L}^1_{K-1,\Gamma}(\omega_i)$. 
We define the operator $C_{\Gamma} \colon L^2(\Omega) \to \mathcal{L}^1_{K,\Gamma}(\mathcal{T})$ by
\begin{align}\label{eq:def_C_D}
  C_{\Gamma} \coloneqq \sum_{i\in \vertices} \varphi_i C_{i,\Gamma}.
\end{align}
We use the same distance~$\delta$ on $\tria$ induced by the nodal neighbor notion, see Definition~\ref{def:Metric}.  
Let us show that~$C_{\Gamma}$ and $\delta$ satisfy~\ref{itm:self-adjoint}--\ref{itm:locality} with $\mathcal{L}^1_K(\mathcal{T})$ replaced by $\mathcal{L}^1_{K,\Gamma}(\mathcal{T})$. 

The proof of~\ref{itm:self-adjoint} (Self-adjoint), \ref{itm:distance} (Distance) and~\ref{itm:locality} (Locality) follows as before. 
The proof of the upper bound $\lambda_{\max}( C_{\Gamma}) \leq 1$ is also the same.

With the same operators~$\frD_i \colon \mathcal{L}^1_K(\mathcal{T}) \to \mathcal{L}^1_{K-1}(\omega_i)$ as in Section~\ref{sec:LocDecomp} we have 
\begin{align}
  v_K = \sum_{i\in \vertices} \varphi_i \frD_i v_K\qquad\text{for all
  }v_K\in \mathcal{L}_K^1(\mathcal{T}).
\end{align}
Due to Remark~\ref{rem:D-dirichlet} the operators $\varphi_i \frD_i$ preserve zero boundary values in the sense that~$\frD_i \colon \mathcal{L}^1_{K,\Gamma}(\mathcal{T}) \to \mathcal{L}^1_{K-1,\Gamma}(\omega_i)$ and~$\varphi_i \frD_i \colon \mathcal{L}^1_{K,\Gamma}(\mathcal{T}) \to \varphi_i \mathcal{L}^1_{K-1,\Gamma}(\omega_i)$. 
In particular, the operators $\varphi_i \frD_i$ are compatible with the function space decomposition~\eqref{eq:space-decomp-D}.
Since we use the same decomposition operators~$\varphi_i \frD_i$ and $\mathcal{L}^1_{K,\Gamma}(\mathcal{T}) \subset \mathcal{L}^1_K(\mathcal{T})$ we have for all $v_K \in \mathcal{L}^1_{K,\Gamma}(\mathcal{T})$ that
\begin{align}
  \sum_{i\in \vertices}\int_\Omega \varphi_i\, |\frD_i
  v_K|^2\dx 
  &\leq 
    K_1 \norm{v_K}_2^2 \quad \text{with }K_1 = \frac{2K+d}{K}.
\end{align}
The arguments in Section~\ref{sec:lower-eigenv-bound} lead to \ref{itm:ellipticity} with
\begin{align*}
 \textup{cond}_2(C_{\Gamma}|_{\mathcal{L}_{K,\Gamma}^1(\mathcal{T})}) \leq  \frac{1}{\lambda_{\min}(C_{\Gamma}|_{\mathcal{L}_{K,\Gamma}^1(\mathcal{T})})} 
 \leq K_1 
 =  \frac{2K+d}{K}.
\end{align*}

Overall, we obtain the same estimates for the operator~$C_{\Gamma}$ and the distance~$\delta$ for the slightly modified space decomposition~\eqref{eq:space-decomp-D}. 
This allows us to extend the results in the subsequent sections to the $L^2$-projection $Q_{\Gamma}$ mapping to $\mathcal{L}_{K,\Gamma}^1(\mathcal{T})$, see Section~\ref{sec:stability-zero-val}. 

\section{Weighted Estimates}
\label{sec:WeightedEst}
In this section we derive stability estimates for~$Q$ in weighted Lebesgue and Sobolev spaces. 
In Section~\ref{sec:weightedL2est} we 
deduce weighted estimates in $L^2$ from the decay estimate in Proposition~\ref{prop:decay}. 
This is the starting point for the proof of weighted $L^p$ (Section~\ref{sec:LpStab}) and Sobolev (Section~\ref{sec:SobEst}) stability. 
Notice that the latter parts  solely rely on the weighted $L^2$-estimate and the distance~$\delta$, that is, the result is independent of the design of $C$ and the decay
estimate itself. 
A similar approach to pass from weighted $L^2$-estimates to $L^p$ and $W^{1,p}$-estimates is used  in~\cite{CrouzeixThomee87,Boman06,ErikssonJohnson95}. 
For simplicity of the exposition we derive our results in the case without zero boundary values. 
Using Section~\ref{sec:dirichl-bound-valu} all results directly transfer to the other case and the corresponding estimates are collected in Subsection~\ref{sec:stability-zero-val}.

\subsection{Weighted $L^2$-Stability}\label{sec:weightedL2est}
We start with proving weighted $L^2$-estimates. 
The proof is based on the decay estimate in Section~\ref{sec:DecayWithC}. 
Throughout this section let $\delta$ be some (geodesic) distance function on $\mathcal{T}$ as in~\ref{itm:distance}.
We use the notation $\max_T f \coloneqq \max_{x \in T} f(x)$, for a function $f$ defined on $T$, and analogously for $\min$, $\esssup$ and $\essinf$. 
\begin{subequations}
  \label{eq:graded-weight}
  \begin{definition}[Grading]\label{def:grading-cont}
    We call a positive function $\rho \in L^1(\Omega)$ a weight with grading $\gamma_\rho \geq 1$ (with respect to the distance~$\delta$), if it satisfies the following conditions for all $T, T' \in \mathcal{T}$: 
    \begin{align}
      \label{eq:graded-weight1}
      \esssup_T \rho &\leq \gamma_\rho  \essinf_T \rho;
      \\
      \label{eq:graded-weight2}
      \esssup_{T'} \rho &\leq \gamma_\rho \esssup_T \rho
                          \quad
                          \text{and} \quad
                          \essinf_{T'} \rho \leq \gamma_\rho \essinf_T
                          \rho \quad \text{if $\delta(T,T') =1$.}
    \end{align}
  \end{definition}
By induction it follows from~\eqref{eq:graded-weight2} that for all $T,T' \in \mathcal{T}$
  \begin{align}
    \label{eq:graded-weight3}
    \esssup_{T'} \rho &\leq \gamma_\rho^{\delta(T,T')} \esssup_T \rho
                        \qquad
                        \text{and} \qquad
                        \essinf_{T'} \rho \leq \gamma_\rho^{\delta(T,T')} \essinf_T \rho.
  \end{align}
\end{subequations}
\begin{remark}[Properties of grading]\hfill
  \label{rem:weight}
  \begin{enumerate}
  \item 
  \label{itm:weight-inverse}
    The inverse of a weight~$\rho$ with grading $\gamma_\rho$ is also a weight with grading~$\gamma_\rho$.  
    \item 
    \label{itm:weight-product}
    If $\rho, \psi$ are weights with grading~$\gamma_\rho$ and $\gamma_\psi$, respectively, then the product $\rho \psi$ is a weight with grading $\gamma_{\rho \psi} \leq \gamma_{\rho} \gamma_{\psi}$. 
  \item \label{itm:weight-const}
    If $\rho$ is a weight with grading~$\gamma_\rho$, then we can define an equivalent piece-wise constant weight $\bar{\rho} \in \mathcal{L}^0_0(\mathcal{T})$ by
    \begin{align*}
      \bar{\rho}|_T &\coloneqq \esssup_T \rho.
    \end{align*}
    Then $\bar{\rho}$ is also a weight with grading~$\gamma_\rho$ and
    \begin{align*}
      \rho &\leq \bar{\rho} \leq \gamma_\rho\, \rho.
    \end{align*}
  \item 
  \label{itm:weight-local-est}
  In particular, for any local estimate, e.g., stability or inverse estimates, the corresponding weighted version for weights $\rho$ with grading $\gamma_{\rho}$ hold, with constant depending on $\gamma_{\rho}$. 
  \item 
  \label{itm:weight-continuous}
    If $\rho$ is a continuous weight, then assumption~\eqref{eq:graded-weight} is satisfied if and only if
    \begin{align*}
      \max_{T \in \mathcal{T}} \frac{\max_T \rho}{\min_T \rho} \leq \gamma_\rho.
    \end{align*}
    Indeed, if $\delta(T,T') =1$, then
    \begin{align*}
      \max_{T'} \rho &\leq \gamma_\rho \min_{T \cap T'} \rho \leq
                       \gamma_\rho \max_T \rho
                       \quad
                       \text{and} \quad
                       \min_{T'} \rho \leq \max_{T \cap T'} \rho \leq
                       \gamma_\rho \min_T \rho.
    \end{align*}
  \end{enumerate}
\end{remark}
A typical example for weights $\rho$ are powers of mesh size functions as introduced in Definition~\ref{def:MeshSizeFunction} below. 
For all weights with grading $\gamma_\rho = 1$, the weighted $L^2$-stability of $Q$ as stated in Theorem~\ref{thm:weighted-L2} below is trivial. 
For $\gamma_\rho >1$ the situation is more delicate. 
\begin{definition}[Worst grading]\label{def:WorstGrading}
  By $\gamma_\textup{max}$ we denote the worst grading in the sense that for all weights $\rho \in L^1(\Omega)$ with grading $\gamma_\rho< \gamma_\textup{max}$ there is a constant $C(\gamma_\rho) < \infty$ with
\begin{align*}
\lVert \rho Q u \rVert_2 \leq C(\gamma_\rho)\,\lVert\rho u \rVert_2\qquad\text{for all }u \in L^2(\Omega).
\end{align*} 
\end{definition}
Based on the decay estimate in Proposition~\ref{prop:decay} we achieve the following lower bound for the worst grading.
\begin{theorem}[Weighted $L^2$-stability]\label{thm:weighted-L2}
  For all weights~$\rho \in L^1(\Omega)$ with grading~$\gamma_\rho < 1/q$ for $q$ as in Proposition~\ref{prop:decay} one has that
  \begin{align*}
    \lVert \rho Q u \rVert_2 \leq \frac{6\gamma_\rho^3}{1-\gamma_\rho q}\,\lVert\rho u \rVert_2\qquad\text{for all }u  \in L^2(\Omega).
  \end{align*}
  In particular, one has that $\gamma_\textup{max} \geq 1/q$. 
\end{theorem}
\begin{proof}
  If $\gamma_\rho=1$, then $\rho$ is a constant and the claim is obvious. 
  Thus, we assume that $\gamma_\rho>1$.
  For a weight $\rho\in L^1(\Omega)$ with grading $\gamma_\rho$ we define the layers
  \begin{align}\label{def:layer}
    L_i \coloneqq \{T \in \tria \colon \gamma_\rho^{i-1} < \max_T \rho  \leq \gamma_\rho^i \} \qquad \text{ for all } i \in \mathbb{Z}. 
  \end{align}
  The layers decompose the triangulation~$\mathcal{T}$ into non-overlapping sub-collections of simplices.
  If $T \in L_i$ and $T' \in L_j$, then by~\eqref{eq:graded-weight} it follows that
  \begin{align*}
    \gamma_\rho^{i-1} &< \max_T \rho \leq
                        \gamma_\rho^{\delta(T,T')} \max_{T'} \rho \leq 
                        \gamma_\rho^{\delta(T,T')+j}.
  \end{align*}
  Thus, $i-1 < \delta(T,T')+j$ or equivalently $i \leq \delta(T,T')+j$. 
  Hence, by symmetry we obtain $\abs{i-j} \leq \delta(T,T')$ for any $T \in L_i$ and any $T' \in L_j$, and thus we have that
  \begin{align}
    \label{eq:dist-Li-Lk}
   \abs{i-j} \leq  \delta(L_i, L_j).
  \end{align}
  Using that $(L_i)_{i \in \mathbb{Z}}$ decompose~$\tria$ and the definition of~$L_i$ in \eqref{def:layer}, we obtain that
  \begin{align}\label{eq:ProofWieghtedl2a}
    \begin{aligned}
      \bignorm{\rho Q u}_2^2
      &= \sum_{j\in\mathbb{Z}}  \norm{\indicator_{L_j} \rho
        Q u}_2^2
     \leq \sum_{j\in\mathbb{Z}} \big(\gamma_\rho^j \norm{\indicator_{L_j}
        Q u}_2\big)^2 
      \\
      &\leq \sum_{j\in\mathbb{Z}} \Big(\gamma_\rho^j \sum_{i\in\mathbb{Z}} \norm{\indicator_{L_j}
        Q(\indicator_{L_{j-i}}u)}_2\Big)^2. 
    \end{aligned}
  \end{align}
 By Proposition~\ref{prop:decay} we have that
  \begin{align*}
    \norm{\indicator_{L_j}
    Q(\indicator_{L_{j-i}}u)}_2
    \leq     2\,q^{\max \set{\delta(L_j,L_{j-i})-1,0}} \norm{\indicator_{ L_{j-i}} u}_2.
  \end{align*}
  In combination with $\abs{i} \leq \delta(L_j,L_{j-i})$ and $\gamma_{\rho}q<1$, from \eqref{eq:ProofWieghtedl2a} it follows that 
  \begin{align*}
    &\norm{\rho Q u}_2 
    \leq 2 \bigg( \sum_{j\in \mathbb{Z}} \Big(\gamma_\rho^j \sum_{i\in \mathbb{Z}}
      q^{\max \set{\abs{i}-1,0}}\norm{\indicator_{L_{j-i}}u}_2\Big)^2  \bigg)^{1/2}
    \\
    &\leq 2 \bigg( \sum_{j\in \mathbb{Z}} \Big(\gamma_\rho^j \sum_{i\in \mathbb{Z}}
      q^{\max \set{\abs{i}-1,0}}\norm{\indicator_{L_{j-i}}
      \gamma_\rho^{i-j+2} \rho u}_2\Big)^2  \bigg)^{1/2}
    \\
    &\leq 2 \gamma_\rho^2 \bigg( \sum_{j\in \mathbb{Z}} \Big(\sum_{i\in \mathbb{Z}} \gamma_\rho^i
      q^{\max\set{\abs{i}-1,0}}\norm{\indicator_{L_{j-i}}
      \rho u}_2\Big)^2  \bigg)^{1/2}
    \\
    &\leq 2 \gamma_\rho^3 \bigg( \sum_{j\in \mathbb{Z}} \Big(\sum_{i\in \mathbb{Z}}
      (\gamma_\rho q)^{\max\set{\abs{i}-1,0}}\norm{\indicator_{L_{j-i}} \rho u}_2\Big)^2  \bigg)^{1/2}
    \\
    &\leq 2 \gamma_\rho^3 \bigg( \sum_{j\in \mathbb{Z}} \Big[\Big(
      \sum_{k \in \mathbb{Z}} (\gamma_\rho q)^{\max \set{\abs{k}-1,0}}\Big) 
      \Big( \sum_{i\in \mathbb{Z}} \Big(
      (\gamma_\rho q)^{\max \set{\abs{i}-1,0}}
      \norm{\indicator_{L_{j-i}}\rho u}^2_2\Big) \Big]\bigg)^{1/2}\\
    &= 2 \gamma_\rho^3
      \Big( \sum_{i\in \mathbb{Z}} 
      (\gamma_\rho q)^{\max \set{\abs{i}-1,0}} \Big)
      \norm{\rho u}_2
      \leq \frac{6 \gamma_\rho^3}{1-\gamma_\rho q}\,
      \norm{\rho u}_2.
  \end{align*}
\end{proof}
\begin{remark}[Decay $q$]\label{rem:Decay}
 Proposition~\ref{prop:lower-bound-comp} and \eqref{eq:ConvSpeed2} show that Proposition~\ref{prop:decay} and therefore also Theorem~\ref{thm:weighted-L2} hold with
  \begin{align}\label{eq:qnew}
    q &\coloneqq q_\textup{new} \coloneqq \frac{\sqrt{2K+d} -
        \sqrt{K}}{\sqrt{2K+d} + \sqrt{K}}.
  \end{align}
  Notice, however, that Theorem~\ref{thm:weighted-L2} utilizes solely the decay estimate of Proposition~\ref{prop:decay}.
\end{remark}
An important tool in the following proofs is the maximal operator $M_\gamma$. 
The construction of~$M_\gamma$ is inspired by the \emph{miracle of extrapolation}. 
Indeed, the operator shares many properties with the Rubio de Francia operator~\cite{CruzUribeMartellPerez11}.
\begin{definition}[Maximal operator $M_{\gamma}$]
  \label{def:M-graded}
  For $\gamma>1$ let the maximal operator $M_\gamma\colon  \mathcal{L}^0_0(\tria) \to \mathcal{L}^0_0(\mathcal{T})$  be defined by
  \begin{align*}
    M_\gamma(v_0)|_T &\coloneqq \max_{T' \in \mathcal{T}} \big( \gamma^{-\delta(T,T')}
                    \bigabs{v_0|_{T'}} \big) \qquad \text{for all $v_0 \in \mathcal{L}^0_0(\mathcal{T})$.}
  \end{align*}
\end{definition}
Later we shall use this operator to show weighted $L^p$ and $W^{1,p}$-stability.
\begin{lemma}[Properties of $M_\gamma$]
  \label{lem:prop-M-gamma}
  For all $v_0 \in \mathcal{L}^0_0(\mathcal{T})$ 
  the following is satisfied: 
  \begin{enumerate}
  \item 
  \label{itm:prop-M-grading} (Grading) 
  The function $M_\gamma(v_0)$ is a weight with grading~$\gamma$.
  \item 
  \label{itm:prop-M-majorant}  (Majorant)
   It holds that $\abs{v_0} \leq M_\gamma(v_0)$ in $\Omega$.
  \item
  \label{itm:prop-M-powers} (Powers)
   For any $s >0$ one has $(M_\gamma(v_0))^s = M_{\gamma^s}(\abs{v_0}^s)$ in~$\Omega$.
  \end{enumerate}
\end{lemma}
\begin{proof}
  Property~\ref{itm:prop-M-majorant} is obvious using~$T'=T$ in the definition of $M_{\gamma}$. 
  Also~\ref{itm:prop-M-powers} is immediate, since~$v_0$ is constant on each simplex $T\in \tria$. 
  It remains to prove~\ref{itm:prop-M-grading}. 
  Since~$v_0$ is piece-wise constant, \eqref{eq:graded-weight1} holds. 
 For $T,T' \in \mathcal{T}$ one has that
  \begin{align*}
    \frac{(M_\gamma(v_0))|_T}{(M_\gamma(v_0))|_{T'}}
    &\leq \frac{ \max_{T''} \big( \gamma^{-\delta(T,T'')}
      \bigabs{v_0|_{T''}} \big)}{  \max_{T''} \big( \gamma^{-\delta(T',T'')}
      \bigabs{v_0|_{T''}} \big)}
    \leq
      \max_{T''} \frac{ \gamma^{-\delta(T,T'')}
      \bigabs{v_0|_{T''}}}{ \gamma^{-\delta(T',T'')}
      \bigabs{v_0|_{T''}} }
    \leq  \gamma^{\delta(T,T')}.
  \end{align*}
  This proves~\eqref{eq:graded-weight2}. 
\end{proof}
\begin{remark}[Decay vs. weighted $L^2$-estimates]
  \label{rem:equivalence-decay-weighted}
  Theorem~\ref{thm:weighted-L2} is derived from the decay estimates in Proposition~\ref{prop:decay}. 
  In particular, for the decay parameter~$q$ we obtain weighted $L^2$-estimates for all weights~$\rho \in L^1(\Omega)$ with grading~$\gamma_\rho < 1/q$.

Let us show the reverse, i.e., that weighted $L^2$-estimates for weights with grading~$\gamma$ imply decay estimates with decay parameter $q\coloneqq 1/\gamma$.

  Let~$\gamma \geq 1$ and let $L,L' \subset \mathcal{T}$ be collections of simplices. 
  Then by Lemma~\ref{lem:prop-M-gamma} $M_\gamma(\indicator_L)$ is a weight with grading~$\gamma$ that satisfies $\indicator_L \leq M_\gamma(\indicator_L)$ and
  \begin{align*}
    \indicator_{L'} M_\gamma( \indicator_L)
    \leq     \gamma^{-\delta(L,L')}.
  \end{align*}
  Thus, for $u \in L^2(\Omega)$ the weighted $L^2$-stability implies that
  \begin{align*}
    \norm{ \indicator_{L}Q (\indicator_{L'} u )}_2
    &\leq  \norm{M_\gamma( \indicator_L) Q (\indicator_{L'} u)}_2
      \leq c\,\norm{M_\gamma( \indicator_L)
      \indicator_{L'} u}_2
      \leq c\,\gamma^{-\delta(L,L')} \norm{
      \indicator_{L'} u}_2.
  \end{align*}
  If $\delta(L,L')=0$, we calculate $\norm{ \indicator_{L}Q (\indicator_{L'} u)}_2 \leq \norm{ Q (\indicator_{L'} u)}_2 \leq \norm{\indicator_{L'} u}_2$. 
  Overall, 
  \begin{align}
    \label{eq:weighted-to-decay}
    \norm{ \indicator_{L}Q (\indicator_{L'} u)}_2
        &\leq \min \bigset{c\,\gamma^{-\delta(L,L')},1}\, \norm{
      \indicator_{L'} u}_2.
  \end{align}
This proves that weighted $L^2$-estimates for weights with grading~$\gamma$ imply decay estimates with decay parameter~$q = 1/\gamma$. 
\end{remark}

\subsection{Weighted $L^p$-Stability}\label{sec:LpStab}
Based on the stability of~$Q$ in weighted $L^2$-spaces we  derive weighted~$L^p$-estimates. 
These estimates require suitable mesh
gradings, characterized by the grading of a so-called mesh size function. 
\begin{definition}[Mesh size function]\label{def:MeshSizeFunction}
 A positive weight function $h\in L^1(\Omega)$ is called mesh size function, if the diameter $h_T \coloneqq \textup{diam}(T)$ is equivalent to $h|_T$ for all $T\in \mathcal{T}$ with hidden constants independent of~$T$, meaning that
  \begin{align}\label{eq:equiConstMeshFct}
    \essinf_T h \eqsim h_T \eqsim \esssup_T h.
  \end{align}
\end{definition}
\begin{lemma}[Volume decay]
  \label{lem:aux1} 
    Let $h \in L^1(\Omega)$ be a mesh size function with grading~$\gamma_h$. 
    If $ \gamma_h^{d} < \gamma$, then there exists a constant $c < \infty$ depending only on the shape regularity of $\mathcal{T}$
    and the (hidden) constants in \eqref{eq:equiConstMeshFct} such that
  \begin{align*}
    \sum_{T' \in \mathcal{T}} \abs{T'}\, \gamma^{-\delta(T,T')}
     &\leq c\,  \frac{\log(\gamma_h)}{\log(\gamma/\gamma_h^d)} \,\abs{T}\qquad\text{for all }T\in \mathcal{T}.
  \end{align*}
  \end{lemma}
\begin{proof}
  Let $T,T' \in \mathcal{T}$, $x\in T$ and $x' \in T'$ and denote $N\coloneqq \delta(T,T')$. 
  The property \ref{itm:distance} of the distance $\delta$ yields the existence of a chain $T_0,\dots, T_N$ of simplices with $T_0=T$, $T_N=T'$ and $\delta(T_{j-1},T_{j})=1$ for $j = 1, \ldots, N$.  
  With the grading of the function $\overline{h}\in \mathcal{L}_0^0(\mathcal{T})$ given by Remark~\ref{rem:weight}\ref{itm:weight-const} and the fact that $h_{T_j} \lesssim \overline{h}|_{T_j}$ we have that 
  \begin{align*}
    \abs{x-x'} &\leq \sum_{j=0}^N h_{T_j} \lesssim \sum_{j=0}^N \overline{h}|_{T_j} \lesssim \overline{h}|_{T} \sum_{j=0}^N \gamma_h^j \lesssim h_{T} \gamma_h^N.
  \end{align*}
  By symmetry we can replace~$h_{T}$ in the last term by $\min \set{h_{T}, h_{T'}}$. 
  Hence, the equivalence of $\overline{h}|_{T}$ and $h_T$ and the grading of $h$ yield that
  \begin{align}
    h_T + h_{T'} + \abs{x-x'} &\lesssim \min \set{h_T,h_{T'}}\,
                                \gamma_h^{\delta(T,T')}. 
  \end{align}
Then, by assumption we have $\epsilon \coloneqq \tfrac{1}{2} (\log(\gamma)/\log(\gamma_h) -d)>0$  and
  \begin{align}
    \label{est:gamma-h}
    \gamma^{-\delta(T,T')} 
    <  
    \gamma_h^{-\delta(T,T')(d+\epsilon)} 
    \lesssim \bigg( \frac{\min \set{h_T,h_{T'}}}{ h_T+h_{T'} +
    \abs{x-x'}} \bigg)^{d+\epsilon}.
  \end{align}
   Using the shape regularity to estimate $\abs{T}$ from below by $h_T^d$ we obtain
  \begin{align*}
    \abs{T}^{-1} \sum_{T' \in \mathcal{T}} \abs{T'}  \gamma^{-\delta(T,T')}
   &\lesssim \, \sum_{T' \in \mathcal{T}} \int_{T'} \frac{1}{\abs{T}} \bigg( \frac{\min
      \set{h_T,h_{T'}}}{ h_T+h_{T'} + 
      \abs{x-x'}} \bigg)^{d+\epsilon}\dx'\\
      &\lesssim \, \sum_{T' \in \mathcal{T}} \int_{T'} \frac{h_{T}^{\epsilon}}{ (h_T+h_{T'} + 
      \abs{x-x'})^{d+\epsilon}}\dx' \lesssim  \epsilon^{-1}.
  \end{align*}
 By definition of $\epsilon$ the claim follows.
\end{proof}
Recall~$M_\gamma$ from Definition~\ref{def:M-graded}. 
\begin{lemma}[$L^p$-stability of $M_\gamma$]
  \label{lem:M-gamma-Lp}
  Let $h \in L^1(\Omega)$ be a mesh size function with grading~$\gamma_h$, let
  $p \in [1, \infty]$ and let $\gamma>1$ be such that
  \begin{align*}
    \gamma &> \gamma_h^{d/p}.
  \end{align*}
  With the $p$-independent constant $c>0$ from Lemma~\ref{lem:aux1} there holds that
  \begin{align}\label{eq:TempInequ}
    \norm{M_\gamma(v_0)}_p &\leq \left( c\, \frac{\log(\gamma_h)}{\log(\gamma^p/\gamma_h^d)} \right) ^{1/p} \norm{v_0}_p\quad\text{for all }v_0 \in \mathcal{L}^0_0(\mathcal{T}).
  \end{align}
 The inequality in \eqref{eq:TempInequ} holds with equality for $p = \infty$ (with  constant $1$). 
\end{lemma}
  \begin{proof}
  Let $v_0 \in \mathcal{L}^0_0(\tria)$.
  
  \textit{Step 1.}
    Let $p = \infty$, then \eqref{eq:TempInequ} follows from 
    \begin{align*}
      \big(M_\gamma(v_0)\big)|_T 
      & \leq \norm{v_0}_\infty\qquad\text{for all }T\in \mathcal{T}.
    \end{align*}
    \textit{Step 2.}
    Let us consider the case~$p=1$ with $\gamma > \gamma_h^d$. 
We find that
    \begin{align*}
      &\norm{M_\gamma(v_0)}_1
      = \sum_{T \in \mathcal{T}} \abs{T}\,
        (M_\gamma(v_0))|_T
     =
        \sum_{T \in \mathcal{T}} \abs{T}        \max_{T' \in \mathcal{T}} \big( \gamma^{-\delta(T,T')}
        \bigabs{v_0|_{T'}}\big)
      \\
      &\quad \leq
        \sum_{T \in \mathcal{T}} \abs{T} \sum_{T' \in \mathcal{T}} \big( \gamma^{-\delta(T,T')}
        \bigabs{v_0|_{T'}}\big)
      =
        \sum_{T' \in \mathcal{T}} \int_{T'} \abs{v_0}\dx \bigg( \sum_{T \in \mathcal{T}} \frac{\abs{T}}{\abs{T'}}\, \gamma^{-\delta(T,T')}
        \bigg).
    \end{align*} 
    By Lemma~\ref{lem:aux1} and $\gamma > \gamma_h^d$ we obtain
    \begin{align*}
      \norm{M_\gamma(v_0)}_1 &\leq c \, \frac{\log(\gamma_h)}{\log(\gamma/\gamma_h^d)}
                               \sum_{T' \in \mathcal{T}} \int_{T'}
                               \abs{v_0}\dx  = c\,  \frac{\log(\gamma_h)}{\log(\gamma/\gamma_h^d)} \norm{v_0}_1.
    \end{align*}
    \textit{Step 3.}
    Now let $p\in(1,\infty)$. 
    Lemma~\ref{lem:prop-M-gamma}\ref{itm:prop-M-powers} yields that
    \begin{align*}
      \norm{M_\gamma(v_0)}_p
      &
        =  
        \norm{ \abs{M_\gamma(v_0)}^p}_1^{1/p} =
        \norm{ M_{\gamma^p}(v_0^p)}_1^{ 1/p}.
    \end{align*}
    Since by assumption $\gamma_h^d < \gamma^p$, we can apply the estimate for~$p=1$ and obtain
    \begin{align*}
      \norm{M_\gamma(v_0)}_p
      &\leq
        \left( c\,  \frac{\log(\gamma_h)}{\log(\gamma^p/\gamma_h^d)} \bignorm{ v_0^p}_1\right)^{1/p}
        =  
         \left(c\, \frac{\log(\gamma_h)}{\log(\gamma^p/\gamma_h^d)}\right)^{1/p} \bignorm{ v_0}_p,
    \end{align*}
    which finishes the proof. 
  \end{proof}                                                                      
\begin{theorem}[Weighted $L^p$-estimate]\label{thm:weighted-Lp}
  Let $\rho \in L^1(\Omega)$ be a weight with grading $\gamma_{\rho}$, let $h \in L^1(\Omega)$ be a mesh size function with grading $\gamma_h$, and let $p \in [1,\infty]$. 
  With $\gamma_{\max}$ as in Definition~\ref{def:WorstGrading} assume that
  \begin{align*}
    \gamma_\rho \gamma_h^{d \abs{\frac12 - \frac 1p}} < \gamma_{\max}.
  \end{align*}
  Then, there exists a constant $c = c(\gamma_{\rho},\gamma_h, d,p,K,\chi_0)< \infty$ (with shape-regularity parameter $\chi_0$ of $\mathcal{T}$) such that 
  \begin{align*}
    \norm{\rho Q u}_p &\leq c\, \norm{\rho u}_p\qquad\text{for all }u \in L^p(\Omega).
  \end{align*}
\end{theorem}
\begin{proof}
  Theorem~\ref{thm:weighted-L2} yields the statement for $p=2$. 
  To prove the claim for all other $p \in [1,\infty]$ we proceed in three steps. 
  Throughout the proof, let $u \in L^p(\Omega)$ for the respective exponent $p\in [1,\infty]$. 

  \textit{Step 1.}   
  Let us consider $p \in (2,\infty)$. 
  We choose $\gamma>1$ such that
  \begin{align*}
  \gamma_{\rho} \gamma_h^{d \left( \frac12 - \frac 1p\right)} <
  \gamma_{\rho} \gamma < \gamma_{\max}.
  \end{align*}
  Let $\overline{\rho} \in \mathcal{L}^0_0(\mathcal{T})$ be defined by $\overline{\rho}|_T \coloneqq \esssup_T \rho$ and analogously $\overline{Q u}|_T \coloneqq \esssup_T \abs{Q u}$ for all $T\in \mathcal{T}$.
   By Remark~\ref{rem:weight}\ref{itm:weight-const} the function $\overline{\rho}$ is also a weight with grading~$\gamma_\rho$ and it satisfies $\rho \leq \overline{\rho} \leq \gamma_\rho\, \rho$.
  The pointwise estimate in Lemma~\ref{lem:prop-M-gamma}\ref{itm:prop-M-majorant} leads to  
  \begin{align*}
    \norm{\overline{\rho} Q u}_p^p 
    &=
      \biggnorm{ \abs{ \overline{\rho} Q u}^{\frac{p-2}{2}}
      (\overline{\rho} Qu )}_2^2
      \leq
      \biggnorm{ \abs{ \overline{\rho} \overline{Qu }}^{\frac{p-2}{2}} (\overline{\rho} Qu)}_2^2
      \leq 
      \biggnorm{M_\gamma\big(\abs{\overline{\rho}  \overline{Qu}}^{\frac{p-2}{2}}\big)\, \overline{\rho} Qu}_2^2.
  \end{align*}
  Lemma~\ref{lem:M-gamma-Lp}\ref{itm:prop-M-grading} and Remark~\ref{rem:weight}\ref{itm:weight-product} show that $M_\gamma\big(\abs{\overline{\rho}  \overline{Qu}}^{\frac{p-2}{2}}\big)\, \overline{\rho}$ is a weight with grading $\gamma \gamma_{\rho}$. 
Since $\gamma \gamma_{\rho}<\gamma_{\max}$, the weighted $L^2$-stability of $Q$ in Theorem~\ref{thm:weighted-L2} yields that
  \begin{align*}
    \norm{\overline{\rho} Q u}_p^p \leq 
     \biggnorm{M_\gamma\big(\abs{\overline{\rho}  \overline{Q u}}^{\frac{p-2}{2}}\big)\, \overline{\rho} Qu}_2^2
    \lesssim
    \biggnorm{M_\gamma\big(\abs{\overline{\rho}
    \overline{Qu}}^{\frac{p-2}{2}}\big)\, \overline{\rho} u}_2^2. 
  \end{align*}
Applying H\"older's inequality we obtain 
  \begin{align}\label{eq:ProofTemp4442}
    \norm{\overline{\rho} Qu}_p^p
      &\lesssim
        \biggnorm{ M_\gamma\big(\abs{\overline{\rho} \overline{Qu}}^{\frac{p-2}{2}}\big)}_{\frac{2p}{p-2}}^{2}  \norm{\overline{\rho} u}_p^2.
  \end{align}
  Since $\overline{\rho} \overline{Qu} \in \mathcal{L}^0_{0}(\mathcal{T})$ and $\gamma_h^{d \left( \frac{2p}{p-2}\right)^{-1} } =  \gamma_h^{d \left(\frac 12 - \frac 1p\right) } <\gamma$,
   Lemma~\ref{lem:M-gamma-Lp} ensures the stability of $M_{\gamma}$. 
   The stability, \eqref{eq:ProofTemp4442} and an inverse estimate yield that
  \begin{align*}
    \norm{\overline{\rho} Qu}_p^p
    \lesssim
    \biggnorm{\abs{\overline{\rho} \overline{Qu}}^{\frac{p-2}{2}}}_{\frac{2p}{p-2}}^{2}  \norm{\overline{\rho} u}_p^2 
    \lesssim \norm{\overline{\rho} \overline{Qu}}_{p}^{p-2}  \, \norm{\overline{\rho} u}_p^2
    \lesssim 
    \norm{\overline{\rho} Q u}_{p}^{p-2}  \, \norm{\overline{\rho} u}_p^2.
  \end{align*}
  The relation between $\rho$ and $\overline{\rho}$ (Remark~\ref{rem:weight}\ref{itm:weight-const}) verifies the claim for $p \in (2,\infty)$.\\[0.5em]
  \textit{Step 2.}  
  Let $p=\infty$. For a weight $\rho$ choose $T \in \mathcal{T}$ such that $\norm{\rho Qu }_{\infty} = \norm{\rho Qu}_{\infty,T}$. 
  Further, let $\gamma>1$ such that $\gamma_{\rho} \gamma_h^{d/2} < \gamma_{\rho} \gamma <\gamma_{\max}$. 
First we shall apply the pointwise estimate in Lemma~\ref{lem:prop-M-gamma}\ref{itm:prop-M-majorant}. 
Then, since $\gamma \gamma_{\rho} < \gamma_{\max}$, the weighted $L^2$-estimate for $Qu$ according to Theorem~\ref{thm:weighted-L2} with weight $M_{\gamma}(\indicator_{T}) \rho$ is available. 
This leads to
\begin{align*}
\norm{ \indicator_{T} \rho Qu }_2
 \leq \norm{M_{\gamma}(\indicator_{T}) \rho Qu}_2 
\lesssim 
\norm{M_{\gamma}(\indicator_{T}) \rho u}_2.
\end{align*}
By H\"older's inequality and the stability of $M_{\gamma}$ in $L^{2}(\Omega)$, provided that $\gamma_h^{d/2}<\gamma$ we find that
\begin{align}\label{est:Lp-inf-1}
\norm{ \indicator_{T} \rho Qu }_2
\lesssim 
\norm{M_{\gamma}(\indicator_{T})}_2 \norm{\rho u}_{\infty} \lesssim 
\abs{T}^{\frac 12} \norm{\rho u}_\infty.
\end{align}
Since by Remark~\ref{rem:weight}\ref{itm:weight-local-est} $\rho$ is a weight, a weighted inverse estimate (with constant depending also on $\gamma_{\rho}$) is at our disposal. 
This and \eqref{est:Lp-inf-1} yield 
\begin{align*}
\norm{\rho Qu }_{\infty} = \max_T \abs{\rho Qu} \lesssim
\left(\dashint_{T} \abs{\rho Qu}^2 \dx \right)^{\frac 12}
\lesssim \norm{\rho u }_{\infty}.
\end{align*}
This shows the claim for $p=\infty$. \\[1em]
\textit{Step 3.}  
  Let $p \in [1,2)$. 
  Recall that by Remark~\ref{rem:weight}\ref{itm:weight-inverse} the inverse $\rho^{-1}$ is a weight with grading $\gamma_{\rho}$. 
  Hence, since $p'\geq 2$ and 
  \begin{align*}
    \gamma_{\rho}\gamma_h^{d\left( \frac 12 - \frac{1}{p'} \right)}
    = \gamma_{\rho} \gamma_h^{d\left( \frac 1p - \frac 12 \right)}
    < \gamma_{\max},
  \end{align*}
  we have that $\norm{\rho^{-1} Qv}_{p'} \lesssim \norm{\rho^{-1} v}_{p'}$, for $v\in L^{p'}(\Omega)$. 
For $w \in L^{p'}(\Omega)$ one has that
\begin{align*}
\skp{\rho Qu}{w} = \skp{Qu}{\rho w} = 
\skp{u}{Q(\rho w)} = \skp{\rho u}{\rho^{-1} Q(\rho w)}.
\end{align*}  
This identity, duality, and the $L^{p'}$-stability with weight $\rho^{-1}$ show that
  \begin{align*}
    \norm{\rho Q u }_p 
    = \sup_{w \in L^{p'}(\Omega)} \frac{\skp{\rho Qu}{w} }{\norm{w}_{p'}} 
      =
      \sup_{w \in L^{p'}(\Omega)} \frac{\skp{\rho u}{\rho^{-1} Q(\rho w)} }{\norm{w}_{p'}}
      \lesssim \norm{\rho u}_p.
  \end{align*}
  This finishes the proof.
\end{proof}

\subsection{Weighted Sobolev Stability}
\label{sec:SobEst}
In the following we deduce weighted $W^{1,p}$-stability from weighted $L^p$-estimates with the help of weighted inverse estimates and approximation properties of the Scott--Zhang interpolation operator. 
Let $h \in L^1(\Omega)$ be a mesh size function (Definition~\ref{def:MeshSizeFunction}) with grading $\gamma_h$. Recall the worst grading $\gamma_{\max}$ as in Definition~\ref{def:WorstGrading}. 
\begin{theorem}[Weighted $W^{1,p}$-estimate]\label{thm:weighted-Sobolev}
 Let $\rho \in L^1(\Omega)$ be a weight with grading $\gamma_{\rho} \geq 1$ and let $p \in [1, \infty]$. Assume that 
  \begin{align*}
    \gamma_\rho \gamma_h^{1+d \abs{\frac12 - \frac 1p}} < \gamma_{\max}.  
  \end{align*}
  Then there exists a constant $c = c(\gamma_{\rho},\gamma_h, d,p,K,\chi_0)< \infty$ (where $\chi_0$ denotes the shape regularity parameter of $\mathcal{T}$) such that
  \begin{align*}
    \norm{\rho \nabla Q u}_p &\leq c\, \norm{\rho \nabla u}_{p}\qquad\text{for all }u \in W^{1,p}(\Omega).
  \end{align*}
\end{theorem}
\begin{proof}
  Let $\PiSZ \colon W^{1,1}(\Omega) \to \mathcal{L}^1_K(\mathcal{T})$ denote the Scott--Zhang interpolation operator introduced in \cite{SZ.1990}. 
  Denote by $\omega_T$ the one-layer neighborhood of a simplex $T \in \tria$, and recall that $h$ is locally equivalent to $h_T \coloneqq \diameter(T)$, for $T\in \tria$. 
The operator $\PiSZ$ satisfies, for all $T \in \tria$ and $v \in W^{1,p}(\Omega)$ the estimates
\begin{align*}
 \norm{h^{-1}( v - \PiSZ(v))}_{p,T} \lesssim 
\norm{\nabla v}_{p,\omega_T}\quad \text{and}\quad
\norm{\nabla\PiSZ(v)}_{p,T} \lesssim 
\norm{\nabla v}_{p,\omega_T}.
\end{align*}
Thanks to Remark~\ref{rem:weight}\ref{itm:weight-local-est} the corresponding weighted local and global estimates are available and the same is true for the inverse estimates.

Because $Q$ is the identity on $\mathcal{L}^1_K(\tria)$, we find that $Q \PiSZ = \PiSZ$. 
Hence, a weighted inverse estimate and the local equivalence of $h$ with $h_T$ yield that
\begin{align*}
\norm{\rho \nabla(Q u)}_p 
&\leq
 \norm{\rho \nabla (Q(u- \PiSZ (u)))}_p +  \norm{\rho \nabla (\PiSZ (u))}_p \\
 &\lesssim  
 \norm{\rho h^{-1} (Q(u- \PiSZ (u)))}_p +  \norm{\rho \nabla (\PiSZ (u))}_p.
\end{align*}
Since $h$ and $\rho$ are weights with grading $\gamma_h$ and $\gamma_\rho$, by Remarks~\ref{rem:weight}\ref{itm:weight-inverse} and~\ref{itm:weight-product} the function $\rho h^{-1}$ is a weight with grading $\gamma_\rho \gamma_h$. 
By assumption we have that
 \begin{align*}
 (\gamma_{\rho} \gamma_h) \gamma_h^{d \abs{\frac 12 - \frac 1p}} < \gamma_{\max}.
 \end{align*}
Hence, by the weighted $L^p$-estimate in Theorem~\ref{thm:weighted-Lp} it follows that 
\begin{align*}
\norm{\rho h^{-1} (Q(u- \PiSZ (u)))}_p \lesssim
\norm{\rho h^{-1} (u- \PiSZ (u))}_p.
\end{align*} 
Finally, on both terms we apply the weighted versions of the above approximation and stability estimates for $\PiSZ$ to obtain 
\begin{align*}
\norm{\rho \nabla(Q u)}_p 
& \lesssim
\norm{\rho h^{-1} (u- \PiSZ (u))}_p +  \norm{\rho \nabla (\PiSZ (u))}_p\\
&\lesssim
\norm{\rho h^{-1} h \nabla u}_p +  \norm{\rho \nabla  u}_p \lesssim \norm{\rho \nabla u}_p.
\end{align*}
\end{proof}

\subsection{Zero Boundary Values}
\label{sec:stability-zero-val}

Recall the $L^2$-projection $Q_{\Gamma}$ mapping onto the space $\mathcal{L}^1_{K,\Gamma}(\tria)$ as introduced in Section~\ref{sec:dirichl-bound-valu}. 
For a resolved subset of the boundary $\Gamma \subset \partial \Omega$ the space $\mathcal{L}^1_{K,\Gamma}(\tria)$ contains the functions in $\mathcal{L}^1_K(\tria)$ with zero traces on $\Gamma$. 
Similarly we define $W^{1,p}_{\Gamma}(\Omega) = \{u \in W^{1,p}(\Omega)\colon u|_{\Gamma}= 0\}$. 
Then we obtain analogous stability results for $Q_{\Gamma}$ as for $Q$. 

\begin{theorem}[Weighted $L^p$- and $W^{1,p}$-stability]\label{thm:Lp-W1p-zero-trace}
  Let $\rho \in L^1(\Omega)$ be a weight with grading $\gamma_{\rho}$, let $h \in L^1(\Omega)$ be a mesh size function with grading $\gamma_h$, let $p \in [1,\infty]$ and let $\gamma_{\max}$ be as in Definition~\ref{def:WorstGrading} replacing $Q$ by $Q_{\Gamma}$. 
 \begin{enumerate}[label=(\roman*)]
 \item \label{itm:wLp-est-QD} ($L^p$-stability) If 
  \begin{align*}
    \gamma_\rho \gamma_h^{d \abs{\frac12 - \frac 1p}} < \gamma_{\max},
  \end{align*}
  then, there exists a constant $c = c(\gamma_{\rho},\gamma_h, d,p,K,\chi_0)< \infty$ (with shape-regularity parameter $\chi_0$ of $\mathcal{T}$) such that 
  \begin{align*}
    \norm{\rho Q_{\Gamma} u}_p &\leq c\, \norm{\rho u}_p\qquad\text{for all }u \in L^p(\Omega).
  \end{align*}
 \item \label{itm:wW1p-est-QD} ($W^{1,p}$-stability) If 
  \begin{align*}
    \gamma_\rho \gamma_h^{1+d \abs{\frac12 - \frac 1p}} < \gamma_{\max},
  \end{align*}
  then, there exists a constant $c = c(\gamma_{\rho},\gamma_h, d,p,K,\chi_0)< \infty$ such that 
  \begin{align*}
    \norm{\rho \nabla Q_{\Gamma} u}_p &\leq c\, \norm{\rho \nabla u}_p\qquad\text{for all }u \in W^{1,p}_{\Gamma}(\Omega).
  \end{align*}  
  \end{enumerate}
\end{theorem}
\begin{proof}
Recall that in Section~\ref{sec:dirichl-bound-valu} we have noted that $C_{\Gamma}$ satisfies the same properties as $C$. 
Then, the decay estimate in Proposition~\ref{prop:decay}, the weighted $L^2$-stability in Theorem~\ref{thm:weighted-L2} and the weighted $L^p$-stability in Theorem~\ref{thm:weighted-Lp} translate, which proves \ref{itm:wLp-est-QD}. 
The proof of $W^{1,p}$-stability follows with the fact that the Scott--Zhang interpolation operator preserves discrete boundary values. 
\end{proof}

\section{Refinement Strategies}
\label{sec:results-diff-refin}
In this section we review several mesh refinement strategies. 
Moreover, we present the range of polynomial degrees~$K$ and exponents~$p$ for which unweighted $L^p$ and $W^{1,p}$-estimates are satisfied due to results in the previous sections.

\subsection{Overview of Refinement Strategies}
\label{subsec:RefStrat}
A number of mesh refinement strategies have been developed for two and higher space dimensions. 
Such strategies are used in particular in adaptive finite element methods (AFEM) to resolve singularities of solutions. 
To ensure $L^p$ and $W^{1,p}$-stability estimates it is important to resort to refinement strategies which guarantee a reasonable grading. 

To achieve a sufficiently small grading it is beneficial to avoid working with the diameters $\diameter(T)$ and use an equivalent mesh size function in the sense of Definition~\ref{def:MeshSizeFunction} instead. 
In fact, the grading of such a mesh size function~$h$ can be significantly smaller than the one of $\sum_{T \in \tria} \diameter(T)\indicator_T\in \mathcal{L}^0_0(\tria)$. 

We compare the following refinement strategies: 
\begin{tabbing}
  \quad \= 2D-RGB \qquad \=  Red-Green-Blue refinement by Carstensen
  \\
  \> 2D-NVB$^+$  \>
  Newest vertex bisection with additional assumption on~$\mathcal{T}_0$
  \\
  \> 2D-NVB$^-$ \qquad \>
  Newest vertex bisection without additional assumption on~$\mathcal{T}_0$
  \\
  \> 2D-RG \> Red-Green Refinement
  \\
  \> 2D-RG-GHS \>  Red-Green refinement modified by
  Gaspoz--Heine--Siebert 
  \\[1.5mm]
  \> BiSec-MT \> Generalization of newest vertex bisection to $d\geq 2$
  \\
  \> BiSecLG$(\alpha)$ \> Modification of BiSec-MT with limited grading for parameter $\alpha \in \mathbb{N}$
\end{tabbing}
A summary of the grading (with respect to the distance $\delta$ as in Definition~\ref{def:Metric}) available for each of those refinement strategies can be found in Table~\ref{tab:grading-overview}. 
\begin{table}[ht]
  \centering
\begin{TAB}(r)[2pt]{|l|c|l|}{|c|c|c|c|c|c|c|c|}
    Refinement Strategy& Grading $\gamma_h$ & Reference
    \\
    2D-RGB & $ 2^{ 3/2}$ & \cite{C.2004}
    \\
    2D-NVB$^+$ & $ 2$ & \cite{GHS.2016}
    \\
    2D-NVB$^-$ & $ 2^{ 3/2}$ & \cite{C.2004,GHS.2016}
    \\
    2D-RG & $4$ & \cite{GHS.2016} 
    \\
    2D-RG-GHS & $2$ & \cite{GHS.2016}
    \\
    BiSec-MT & $(2)$ & Conjecture~\ref{conj:NVB-grad}
    \\
    BiSecLG$(\alpha)$ & $2^{\alpha / d}$ &
    Section~\ref{sec:grad-pres-nvb}
    \\
  \end{TAB}
  \caption{Grading obtained by refinement strategies}  \label{tab:grading-overview}
\end{table}
Let us give some more insight into the grading estimates. 
Starting from an initial triangulation~$\mathcal{T}_0$ the refinement strategy is applied repeatedly. 
Usually, in the analysis of the mesh size function the initial triangulation $\mathcal{T}_0$ is assumed to be uniform (with possibly large constant). 
In particular, the mesh size function $h_0$ of~$\mathcal{T}_0$ is chosen as constant. 
This affects the equivalence constants of the mesh size function~$h$, see~\eqref{eq:equiConstMeshFct}.  
Then it is shown that the refinement rules ensure a limited grading. 

\textbf{(2D-RGB)} 
In \cite{C.2004} Carstensen introduces a refinement strategy based on red, green and blue refinement steps. 
He was the first to work with a regularized nodal mesh size function~\cite{Carstensen02}. 
The construction first assigns to each node~$j$ the value~$2^{-\ell(j)}$, where $\ell(j)$ it the highest refinement level of the adjacent triangles. 
Then he regularizes this nodal function using the edge based distance function. 
This approach establishes a mesh size function with grading~$\gamma_h \leq 2^{3/2}$, see \cite[Theorem~4.1]{Carstensen02}.

\textbf{(2D-NVB$^+$)} 
The newest vertex bisection was introduced by Mitchell in \cite{Mitchell91} for suitable initial triangulations~$\mathcal{T}_0$. 
In particular, it is assumed that each simplex of~$\mathcal{T}_0$ has a matching reflected neighbor, cf.\ \cite{Stevenson08}. 
In~\cite[Theorem~3.1]{GHS.2016} Gaspoz, Heine and Siebert obtain a mesh size function with grading~$\gamma_h \leq 2$ using a regularized mesh size function similar
to the one of~\cite{Carstensen02}.

\textbf{(2D-NVB$^{-}$)} 
Gaspoz, Heine and Siebert also study the newest vertex bisection of Mitchell \cite{Mitchell91} without additional
assumptions on~$\mathcal{T}_0$. 
In that case they work with a mesh size function with grading~$\gamma_h \leq 2^{3/2}$, see~\cite[Section~5.2]{GHS.2016}.

\textbf{(2D-RG)} 
Further, in \cite[Section~5.3]{GHS.2016} Gaspoz, Heine and Siebert consider the red-green refinement and construct a mesh function with grading~$\gamma_h=4$.

\textbf{(2D-RG-GHS)} 
Gaspoz, Heine and Siebert also introduce a modified red-green refinement that avoids green neighbors~\cite[Section~5.4]{GHS.2016}. 
By this they obtain a mesh size function with grading~$\gamma_h=2$.

To the best of our knowledge there are no grading estimates available for any adaptive refinement scheme in higher space dimensions. 

\textbf{(BiSec-MT)}
Maubauch \cite{Maubach95} and Traxler \cite{Traxler97} develop a generalization of the newest vertex bisection to higher space dimensions $d\geq 2$, see also  \cite{Stevenson08}. 
Private communication with Fernando Gaspoz \cite{G.2020} leads to the following conjecture\footnote{The same conjecture was formulated by one of the referees.}.  
\begin{conjecture}[Grading of BiSec-MT]\label{conj:NVB-grad} 
For dimensions $d \geq 2$ let $\tria_0$ be an initial triangulation satisfying the matching neighbor condition by Stevenson~\cite{Stevenson08}. 
Then, for any triangulation generated by the BiSec-MT algorithm there exists a mesh size function $h \in L^1(\Omega)$ with grading $\gamma_h =2$ for which the constants in \eqref{eq:equiConstMeshFct} only depend on $\tria_0$. 
\end{conjecture}

\textbf{(BiSecLG$(\alpha)$)}
In the following Section \ref{sec:grad-pres-nvb} for $\alpha \in \mathbb{N}$ we modify the BiSec-MT algorithm in such a way that there exists a mesh size function with grading $\gamma_h = 2^{\alpha/d}$. 
\subsubsection*{Stability Results} 
Let us discuss the stability results obtained for meshes with a given grading $\gamma_h$ for suitable choices of exponents $p$ and polynomial degrees $K$, depending on the dimension. 
Recall that with $q_\textup{new}$ as defined in \eqref{eq:qnew} we have the following lower bound on the worst grading  
\begin{align*}
\gamma_{\max} \geq \frac{1}{q_\textup{new}} = \frac{\sqrt{2K + d} + \sqrt{K}}{\sqrt{2K+d}-\sqrt{K}},
\end{align*}
as in \eqref{def:WorstGrading}. 
Thus, Theorems~\ref{thm:weighted-Lp} and \ref{thm:weighted-Sobolev} ensure (unweighted) $L^p$ and $W^{1,p}$-stability, provided that
\begin{align}\label{est:cond-stab}
\gamma_h^{d\abs{\frac 12 - \frac 1p}} < q_\textup{new}^{-1} \quad \text{ and }\quad
\gamma_h^{1+d\abs{\frac 12 - \frac 1p}} < q_\textup{new}^{-1}, 
\end{align}
respectively.  
The Tables~\ref{tab:grading-stability-2D} and \ref{tab:grading-stability-3D} present the ranges of $p$, for which $L^p$ and $W^{1,p}$-stability holds for given grading $\gamma_h$ and degree $K$, in dimensions $d\in \{2,3\}$. 
Note that the range of admissible exponents $p$ increases in $K$ for fixed $\gamma_h$ and $d$. 
\begin{table}[ht]
  \centering
\begin{TAB}(r,0.4cm,0.5cm)[2pt]{|c|c|c|c|l|}{|cc|c|cccc|cccc|cccc|}
    Grading  & Degree  &$L^p$-stability&$W^{1,p}$-stability & Refinement
    \\
    $\gamma_h $ &  $K $ & for $p \in$ & for $p \in $ & Strategies\\
        $ 2^{1/2}$ & $ 1$ & $[1,\infty]$ & $ [1,\infty]$ & BiSecLG$(1)$
    \\
     $ 2$ & $ 1$ & $[1,\infty]$ & $ [1.2619,4.8188]$ & NVB$^+$, RG-GHS
    \\
      & $2$ & $[1,\infty]$ & $ [1.0527,19.9937] $ & 
    \\
      & $3$ & $[1,\infty]$ & $[1,\infty]$ & 
    \\
      & $\infty$& $[1,\infty]$ & $[1,\infty]$ &  
	\\
     $ 2^{3/2}$ & $1$& $[1,\infty]$ & $ [1.8928,2.1200]$ & RGB, NVB$^-$ 
    \\
      & $2$& $[1,\infty]$ & $[1.5790,2.7271] $ &  
    \\    
      & $3$& $[1,\infty]$ & $ [1.4589,3.1794]$ &  
    \\       
    & $\infty$& $[1,\infty]$ & $ [1.1797, 6.5660]$ &  
	\\
     $ 4$ & $1$ & $ [1.1158,9.6376]$ & $\emptyset$ & RG 
    \\
      & $2$ & $ [1.0257,39.9874]$ & $\emptyset$ &  
    \\
      & $3$ & $[1,\infty]$ & $ [1.9452,2.0580]$ &  
    \\
        & $\infty$& $[1,\infty]$ & $[1.5729, 2.7455]$ &  
\\
  \end{TAB}
  \caption{Guaranteed stability estimates in 2D}  \label{tab:grading-stability-2D}
\end{table}

\begin{table}[ht]
  \centering
\begin{TAB}[3pt]{|c|c|c|c|l|}{|cc|c|cccc|} 
\hline
    Grading  & Degree  &$L^p$-stability&$W^{1,p}$-stability & Refinement 
    \\
    $\gamma_h $ &  $K $ & for $p \in$ & for $p \in $ & Strategies\\
     $ 2^{1/3}$ & $ 1$ & $[1,\infty]$ & $ [1,\infty]$ & BiSecLG$(1)$ \\
     $ 2$ & $ 1$ & $[1.0387,26.9019]$ & $ [1.5886,2.6990]$ & BiSec-MT (conjecture) and 
    \\
      & $2$ & $[1,\infty]$ & $ [1.3508,3.8511] $ & BiSecLG$(3)$
    \\
      & $3$ & $[1,\infty]$ & $[1.2501,4.9997]$ & 
    \\
          & $\infty$ & $[1,\infty]$ & $[1,\infty]$ & 
   \\
  \end{TAB}
  \caption{Guaranteed stability estimates in 3D}  \label{tab:grading-stability-3D}
\end{table}
In particular, considering \eqref{est:cond-stab} we have the special cases of 
\begin{enumerate}
\item $L^{1}$ and $L^{\infty}$-stability, if $\gamma_h^{d/2} < q_\textup{new}^{-1}$,
\item 
$W^{1,2}$-stability, if $\gamma_h < q_\textup{new}^{-1}$,
\item $W^{1,1}$ and $W^{1,\infty}$-stability, if $\gamma_h ^{1 + d/2} < q_\textup{new}^{-1}$.
\end{enumerate}
Let us discuss our results in the context of those ones in the literature with comparable assumption on the grading. 
Recall that the results on $L^p$ and $W^{1,p}$-stability in Section~\ref{sec:WeightedEst} are independent of the construction in the previous sections. 
Thus, we can determine the admissible range of parameters by considering the respective values of the decay $q$, cf. Table~\ref{tab:ComparBY} or bounds on $\gamma_{\max}$ obtained in the literature. 

\begin{table}[ht]
  \centering
\begin{TAB}[3pt]{|c|c|c|l|}{|c|ccc|cccc|cc|}%
    Grading $\gamma_h  $ & Reference & Range of $K $ & Refinement Strategies
    \\
     $ 2$                    &
                   \cite{BramblePasciakSteinbach02,C.2004} & $ \{1\}$
                   & NVB$^+$, RG-GHS
                   \\
 & \cite{BanYse14, GHS.2016} & $ \{1,2,\dots , 12\}$ & 
    \\
      & Theorem~\ref{thm:weighted-Sobolev} & $\{1,2,\dots\}$ & 
       \\
                   $ 2^{3/2}$ & \cite{BramblePasciakSteinbach02,C.2004} & $ \{1\}$ & RGB, NVB$^-$
    \\
      &  \cite{BanYse14, GHS.2016} & $ \{1, 3,4, \ldots, 9 \}$ & \\  
       &  \cite{GaspozHeineSiebert19pre} & $ \{1,2,3,4\}$ & \\  
         & Theorem~\ref{thm:weighted-Sobolev} & $\{1,2,\dots\}$ & 
       \\ 
    $4$ & \cite{GaspozHeineSiebert19pre} & $\{2,3,4\}$ & RG\\
      & Theorem~\ref{thm:weighted-Sobolev} & $\{3,4,\dots\}$ & 
    \\
  \end{TAB}
  \caption{Guaranteed range of $K$ for $W^{1,2}$-stability in 2D}
    \label{tab:H1stability-2D}
\end{table}
\begin{table}[ht]
  \centering
\begin{TAB}[3pt]{|c|c|c|l|}{|c|cc|cc|}
 Grading $\gamma_h  $ & Reference & Range of $K $ & Refinement Strategies \\
  $ 2^{1/3}$ & \cite{BanYse14}
   & $\{1,2,\dots, 13 \}$ & BiSecLG$(1)$ \\
       & Theorem~\ref{thm:weighted-Sobolev}  & $\{1,2,\dots \}$ & \\
        $ 2$ & \cite{BanYse14} & $\{1,2,\dots,7 \}$ & BiSec-MT (conjecture) and
    \\
      & Theorem~\ref{thm:weighted-Sobolev} & $\{1,2,\dots\}$ & 
       BiSecLG$(3)$
       \\
\end{TAB}
  \caption{Guaranteed range of $K$ for $W^{1,2}$-stability in 3D}
    \label{tab:H1stability-3D}
\end{table}
Tables~\ref{tab:H1stability-2D} and \ref{tab:H1stability-3D} compare the ranges of $K$ for which $W^{1,2}$-stability is achieved for given grading $\gamma_h$. 
Note that the results contained in \cite{BanYse14,GHS.2016} can be read off from the values of the decay $q_{\textup{BY}}$ contained in Table~\ref{tab:ComparBY}. 
The comparison shows that our estimates improve existing results for all $d\geq 2$ and $K\in \mathbb{N}$, except of the case $K=2$, dimension $d=2$ and $\gamma_h = 4$ covered in \cite{GaspozHeineSiebert19pre}. 
More specifically, for $d=2$ we have $W^{1,2}$-stability for all polynomial degrees $K \in \mathbb{N}$ whenever $\gamma_h\leq 2^{3/2}$. 
This includes in particular triangulations obtained by the newest vertex bisection. 
For $d=3$ assuming the conjectured grading $\gamma_h =  2$ for the BiSec-MT algorithm, we obtain $W^{1,2}$-stability for all $K  \in \mathbb{N}$. 
Since this grading has not been proved yet in the following subsection we show that the modification BiSecLG$(d)$ of BiSec-MT enforces grading $\gamma_h = 2$. 
For dimension $d\leq 6$ this yields   
 $W^{1,2}$-stability for all $K\in \mathbb{N}$. 

For general $p \in [1,\infty]$, we are not aware of any results on $L^p$ and $W^{1,p}$-stability for realistic gradings for $d \in \{2,3\}$.  
This includes the newest bisection refinement and variants thereof.  
For $d = 2$ and $\gamma_h \leq 2^{3/2}$ (including the newest vertex bisection) we obtain $L^{\infty}$-stability for any $K \geq1$, see Table~\ref{tab:grading-stability-2D}. 
For $\gamma_h = 2$ (including the case NVB$^+$) $W^{1,\infty}$-stability is satisfied for all $K \geq 3$. 
Furthermore, in dimension $d=3$ we have $L^{\infty}$-stability if $\gamma_h \leq 2$ for all $K \geq 2$, see Table~\ref{tab:grading-stability-3D}. 
This includes both the BiSec-MT algorithm for the conjectured grading $\gamma_h = 2$ as well as the BiSecLG$(\alpha)$ algorithm for $\alpha = 3$ presented in the following subsection.  
All remaining stability results can be obtained using the BiSecLG$(1)$ algorithm leading to grading $\gamma_h^{1/d}$, see \eqref{eq:grad-BiSecLG} below. 
Indeed, Table~\ref{tab:grading-stability-2D} shows that in 2D $W^{1,\infty}$-stability holds also for $K=1,2$, and Table \ref{tab:grading-stability-3D} shows that in 3D both $L^{\infty}$-stability is satisfied for $K=1$ and $W^{1,\infty}$-stability holds for any $K \in \mathbb{N}$. 

\subsection{Bisection with Limited Grading (BiSecLG)}
\label{sec:grad-pres-nvb}
In this section we propose a variant of the BiSec-MT algorithm that guarantees a reduced grading in general dimensions $d \geq 2$. 
Analogously as for the modification in \cite{DemlowStevenson2011} each refinement cycle is accompanied by an additional refinement enforcing a grading condition on the refined triangulation. 
Let us assume that~$\mathcal{T}_0$ is a given initial triangulation satisfying the matching neighbor condition by Stevenson~\cite[Section~4]{Stevenson08}. 
An important feature of the BiSec-MT algorithm is the fact that the number of additional simplices generated by repeated refinements can be controlled by the number of simplices marked for refinement. 
For $d=2$ dimensions the proof goes back to Binev, Dahmen and DeVore~\cite{BinDahDeV04}. 
In \cite{Stevenson08} Stevenson generalizes the result to all dimensions $d \in \mathbb{N}$. 
More specifically, let $(\mathcal{T}_n)_{n \in \mathbb{N}_0}$ denote a sequence of triangulations obtained by the refinement strategy which consists of marking simplices for refinement, bisecting and applying the conformal closure routine. 
Let $\mathcal{M}_{m} \subset \mathcal{T}_m$ for $m\in \mathbb{N}_0$ denote the set of simplices marked for refinement, then
\begin{align}\label{est:NVB-BDD}
\# \mathcal{T}_n - \# \mathcal{T}_0 \lesssim \sum_{m=0}^{n-1} \# \mathcal{M}_m\quad \text{ for all } n \in \mathbb{N}. 
\end{align} 
This estimate is essential in the proof of optimal convergence rates in adaptive finite element methods, cf.~\cite{Stevenson07,CarstensenFeischlPagePraetorius14,DieningKreuzerStevenson16}. 

For $\alpha \in \mathbb{N}$ we  propose the modification BiSecLG$(\alpha)$ of BiSec-MT producing a mesh size function~$h$ with grading~$\gamma_h = 2^{\alpha/d}$ in $d \geq 2$ dimensions while preserving estimate \eqref{est:NVB-BDD}, analogously as done in \cite{DemlowStevenson2011}. 
Let us first describe the refinement algorithm. 
We proceed as in~\cite{Stevenson08} and start from an initial triangulation~$\mathcal{T}_0$ satisfying the matching neighbor condition. 
All simplices of~$\mathcal{T}_0$ have level $\ell = 0$ and 
each bisection of a simplex increases its level by one. 
We modify the refinement routine of Stevenson such that all simplices that touch have levels differing by
at most $\alpha$ for fixed $\alpha \in \mathbb{N}$. 
More precisely, we say that~$\mathcal{T}$ arising from $\tria_0$ through bisection has $\alpha$-\textit{limited grading} if
\begin{align}
  \label{eq:LG}
  \abs{\ell(T) - \ell(T')} \leq \alpha \qquad \text{for all $T,T' \in
  \mathcal{T}$ with $T \cap T' \neq \emptyset$.}
\end{align}
Note that while each bisection halves the volume, $d$ uniform bisection steps are required to halve the diameter. 
If $\mathcal{T}$ arises from $\tria_0$ through bisections and has limited grading \eqref{eq:LG}, then the local mesh size
function~$h|_T \coloneqq 2^{-\ell(T)/d}$ admits the grading
\begin{align}\label{eq:grad-BiSecLG}
  \gamma_h &\leq 2^{\alpha/d}.
\end{align}
Condition \eqref{eq:LG} is the point in which BiSecLG$(\alpha)$ differs from the modification in \cite{DemlowStevenson2011}. 
In fact their algorithm ensures a limited grading with grading notion based on the Euclidean distance between simplices, as used in Lemma~\ref{lem:DistStev} below. 
Note that they obtain validity of \eqref{est:NVB-BDD} by similar arguments as the ones used in the following. 

Let $\texttt{refine}(\mathcal{T},T)$ denote the triangulation arising from Stevenson's refinement routine, which bisects the simplex~$T \in \mathcal{T}$ and applies the conformal closure. 
Let $\texttt{refine}(\mathcal{T},\mathcal{M})$ be the smallest conforming refinement of $\mathcal{T}$ such that the intersection with $\mathcal{M} \subset \mathcal{T}$ is empty. 
Note that this refinement results from a successive application of $\texttt{refine}(\mathcal{T},T)$ for $T \in \mathcal{M}$, see Theorem~5.1 and Section~6 in \cite{Stevenson08}. 
The following algorithm displays the modified refinement routine $\texttt{refine-LG}(\mathcal{T},\mathcal{M})$ for some $\alpha \in \mathbb{N}$. 

\begin{algorithm}[H]
  \SetAlgoLined
  \TitleOfAlgo{Bisection with Limited Grading (BiSecLG$(\alpha)$):}
  \KwData{Partition~$\mathcal{T}_0 \coloneqq \mathcal{T}$, marked simplices~$\mathcal{M}_0 \coloneqq \mathcal{M} \subset \mathcal{T}$ and $m\coloneqq 0$}
  \KwOut{Smallest conforming refinement of~$\mathcal{T}$ with limited grading~\eqref{eq:LG}}
  \Repeat{$\mathcal{M}_{m}=\emptyset$}{
    Increase $m \coloneqq m+1$; \\
	Set the triangulation $\mathcal{T}_{m} \coloneqq\texttt{refine}(\mathcal{T}_{m-1}, \mathcal{M}_{m-1})$; 
	\hfill\textbf{(closure step)}
 \\
    Define the set $\mathcal{M}_{m}$ of all~$T \in \mathcal{T}_{m}$ with $\ell(T) < \ell(T')- \alpha$ and $T \cap T' \neq \emptyset$ for some $T'\in \mathcal{T}_m$;  \hfill
    \textbf{(grading control)}
        }  
\Return{$\textup{\texttt{refine-LG}}(\mathcal{T},\mathcal{M})\coloneqq \mathcal{T}_{m}$}
\end{algorithm}

In the remainder of this section we show that the BiSecLG algorithm preserves \eqref{est:NVB-BDD} and that the output additionally satisfies \eqref{eq:LG}. 
Let $\tria$ be a regular triangulation resulting from the successive application of the \texttt{refine} routine to an initial triangulation $\tria_0$ that satisfies the matching neighbor condition. 
\begin{lemma}[Properties of BiSecLG$(\alpha)$]
\label{lem:PropNVB-LG}
For any set of marked simplices $\mathcal{M} \subset \mathcal{T}$ the algorithm BiSecLG$(\alpha)$ terminates. 
The resulting triangulation $\texttt{refine-LG}(\tria,\mathcal{M})$ is the smallest conforming refinement of $\tria$ in which all elements of $\mathcal{M}$ are bisected and \eqref{eq:LG} is satisfied. 
Furthermore, if $\mathcal{M}$ consists of a single simplex $T \in \tria$, then any newly generated simplex $T' \in \texttt{refine-LG}(\tria,T)\setminus \mathcal{T}$ satisfies that
\begin{align}\label{est:level-refine}
\ell(T') \leq \ell(T) + 1.
\end{align}
\end{lemma}
\begin{proof}
Since the \texttt{refine} routine terminates \cite[Thm.\ 5.1]{Stevenson08}, it remains to show that the loop in the BiSecLG$(\alpha)$ algorithm terminates after finitely many steps $M\in \mathbb{N}$. 
Theorem~5.1 in \cite{Stevenson08} ensures that for all $m\in\mathbb{N}$ we have that
\begin{align*}
\ell(T') \leq \max_{T\in \mathcal{M}_{m-1}} \ell(T) + 1\qquad\text{for all } T' \in \mathcal{T}_m\setminus \mathcal{T}_{m-1}.
\end{align*}
This and \eqref{eq:LG} show that the levels of all simplices in~$\mathcal{M}_m$ are at least by the integer $\alpha$ smaller than the maximal level of simplices contained in~$\mathcal{M}_{m-1}$ in the preceding closure step. 
In particular, there exists a number $M\in \mathbb{N}$ with 
\begin{align*}
&M \leq 1 + \max_{T\in \mathcal{M}} \ell(T) / \alpha,
\end{align*}
such that $\mathcal{M}_M = \emptyset$, and hence the algorithm terminates. 
With the same arguments also estimate \eqref{est:level-refine} is proved. 
Since the \texttt{refine} routine leads to the smallest conforming triangulation, induction shows that $\texttt{refine-LG}(\tria,\mathcal{M})$ is the smallest conforming refinement of $\tria$ in which all simplices in $\mathcal{M}$ are bisected and \eqref{eq:LG} holds. 
\end{proof}
The following lemma states an estimate similar to the one in \cite[Thm.~5.2]{Stevenson08}. 
By $\dist$ we denote the Euclidean distance in $\mathbb{R}^d$.

\begin{lemma}[Distance]\label{lem:DistStev}
Any new simplex $T' \in \texttt{refine-LG}(\tria,T)\backslash \tria$ satisfies 
\begin{align*}
\dist(T,T') \lesssim  2^{-\ell(T')/d}.
\end{align*}
The hidden constant depends only the initial triangulation $\tria_0$, $d$ and $\alpha$. 
\end{lemma}
\begin{proof}
First note that (4.1) in \cite{Stevenson08} states that 
\begin{align}\label{eq:h-equiv-NVB}
  2^{-\ell(T)} \eqsim h_T^d,
\end{align}
with constants depending only on the shape regularity of $\tria_0$. 
This property remains unchanged, since we only use the routine $\texttt{refine}$ by Stevenson repeatedly. 

Theorem~5.2 in~\cite{Stevenson08} states that 
any newly created~$T' \in \texttt{refine}(\mathcal{T},T) \setminus \mathcal{T}$ satisfies (with hidden constant only depending on the initial triangulation $\mathcal{T}_0$) that
\begin{align}
  \label{eq:51stevenson}
  \dist(T',T) &\lesssim 2^{-\ell(T')/d}.
\end{align}
Then by design of our algorithm we find simplices $T_1,\dots, T_N$ and $T_1', \dots, T_N'$ with $T_1 = T$ and $T_N'=T'$ such that $T_j'$ is created by a call of~$\texttt{refine}(\cdot,T_j)$, $T_j' \cap T_{j+1} \neq \emptyset$ and $\ell(T_{j+1}) \leq \ell(T_j') - \alpha-1$, for all $j = 1,\ldots,N-1 $. 
This and the fact that $\ell(T_j') \leq \ell(T_j)+1$ by Stevenson proves for any $1 \leq i \leq j \leq N$ one has that
\begin{align*}
  \ell(T_j') \leq \ell(T_j) + 1  \leq \ell(T_{j-1}')-d \leq \dots \leq
  \ell(T_i') -(j-i)\alpha.
\end{align*}
In particular this shows that 
\begin{align}\label{est:level}
\ell(T'_i) &\geq \ell(T'_N) + (N-i)\alpha. 
\end{align}
With~\eqref{eq:51stevenson} it follows that
\begin{align*}
\dist(T'_j, T_j) \lesssim 2^{-\ell(T'_j)/d} \quad \text{ for all } j = 1,\ldots, N. 
\end{align*}
Using this, \eqref{eq:h-equiv-NVB},  
the fact that $\ell(T_j') \leq \ell(T_j)+1$ and \eqref{est:level}, with hidden constants independent of $N$ we obtain that
\begin{align*}
  \dist(T_N',T_1)
  &\leq \sum_{j=1}^N \dist(T_j', T_j)
    + \sum_{j=1}^{N-1} (h_{T_j'} + h_{T_{j+1}}) 
  \\
  &\lesssim \sum_{j=1}^N 2^{-\ell(T'_j)/d} 
    + \sum_{j=1}^{N-1} 2^{-\ell(T'_j)/d} +  \sum_{j=2}^{N} 2^{-\ell(T_{j})/d}  
    \\
     &\lesssim
      \sum_{j=1}^N 2^{-\ell(T'_j)/d} 
    +  \sum_{j=2}^{N} 2^{-(\ell(T'_j)-1)/d}
  \\
   &\lesssim
      \sum_{j=1}^N 2^{-\ell(T'_j)/d} 
        \lesssim
      \sum_{j=1}^N 2^{-(\ell(T'_N)/d + (N-j)\alpha/d)} 
      \\
  &\lesssim 2^{-\ell(T'_N)/d} \sum_{j=1}^N 
  2^{-(N-j)\alpha/d} 
  \lesssim
   2^{-\ell(T'_N)/d} \sum_{i=0}^{N-1} 
  (2^{\alpha/d})^{-i} \\
  &\lesssim 2^{-\ell(T'_N)/d}.
\end{align*}
This proves the claim since $2^{\alpha/d}>1$. 
\end{proof}
\begin{theorem}[Closure estimate]
\label{thm:Control}
Let~$(\mathcal{T}_m)_{m\in \mathbb{N}}$ be a
sequence of refinements obtained by the successive application of the BiSecLG$(\alpha)$ algorithm for some initial triangulation $\mathcal{T}_0$ with the matching neighbor condition. 
Let~$\mathcal{M}_m \subset \mathcal{T}_m$ denote the sets of simplices marked for refinement, then
\begin{align}
  \label{eq:BDD-LG}
  \# \mathcal{T}_n - \# \mathcal{T}_0 \lesssim \sum_{m=0}^{n-1} \# \mathcal{M}_m\qquad\text{for all }n\in \mathbb{N}.
\end{align}
\end{theorem}
\begin{proof}
In~\cite{Stevenson08} Stevenson states that this property only relies on the estimate in~(4.1) and Theorems~5.1 and 5.2 therein. 
As argued in the proof of Lemma~\ref{lem:DistStev} the equivalence (4.1) in \cite{Stevenson08} is still satisfied, cf. \eqref{eq:h-equiv-NVB}. 
Lemma \ref{lem:PropNVB-LG} and \ref{lem:DistStev} correspond to Theorem~5.1 and 5.2 in \cite{Stevenson08}, respectively. 
Hence, Theorem~6.1 of~\cite{Stevenson08} and the remark thereafter result in \eqref{eq:BDD-LG}.
\end{proof}

\section{Crouzeix--Raviart Elements}
\label{sec:BY_CR}
In this section we investigate the stability of the $L^2$-projection onto the Crouzeix--Raviart finite element space. 
Let $\mathcal{F} = \mathcal{F}(\tria)$ denote the set of faces ($(d-1)$-simplices) of simplices in $\tria$ and let $\mathrm{mid}(f)$ denote the midpoint (center of mass) of $f\in \mathcal{F}$. 
For $d\geq 2$ the finite element space is defined by
\begin{align*}
  \CR \coloneqq \lbrace v_\textup{CR} \in \mathcal{L}_1^0(\mathcal{T}) \colon v_\textup{CR}
  \text{ is continuous in } \textup{mid}(f) \text{ for all }f\in \mathcal{F}\rbrace. 
\end{align*}
Let $\QCR\colon L^2(\Omega) \to \CR$ denote the $L^2$-projection, that is, 
\begin{align*}
  \langle \QCR u,v_\textup{CR}\rangle = \langle u,v_\textup{CR}\rangle\qquad\text{for all }u\in L^2(\Omega)\text{ and } v_\textup{CR}\in \CR.
\end{align*}
We derive a decay estimate for the $L^2$-projection $\QCR$ with ideas by \cite{BanYse14}. 
In particular, we design an operator $C= \CCR$ and a distance $\delta=\deltaCR$ with \ref{itm:self-adjoint}--\ref{itm:locality}.

For all $f\in \mathcal{F}$ let $\psi_f\in \CR$ denote the basis function with $\psi_f(\textup{mid}(f)) = 1$ and $\psi_f(\textup{mid}(g)) = 0$ for all faces $g\in \mathcal{F}\setminus \lbrace f \rbrace$.  
Let the (local) operator $C_f \colon L^2(\Omega)\to \CR$ be defined by
\begin{align}\label{eq:def_C_f}
  C_f u \coloneqq \frac{\langle u,\psi_f\rangle}{\langle
  \psi_f,\psi_f\rangle} \psi_f\qquad\text{for all }f\in
  \mathcal{F}\text{ and }u\in L^2(\Omega). 
\end{align}
This operator satisfies that
\begin{align*}
  \langle C_f u ,\psi_f\rangle = \langle u,
  \psi_f\rangle\qquad\text{for all }f\in\mathcal{F}\text{ and }u\in
  L^2(\Omega). 
\end{align*}
Further, let us consider the operator $\CCR \colon L^2(\Omega)\to \CR$ determined by 
\begin{align*}
  \CCR u\coloneqq \sum_{f\in \mathcal{F}} C_f u\qquad\text{for all }u\in L^2(\Omega).
\end{align*}
\begin{lemma}[Self-adjoint]
The operator $\CCR$ is self-adjoint, i.e.,~\ref{itm:self-adjoint} holds.
\end{lemma}
\begin{proof}
  For all $u,w\in L^2(\Omega)$ one has that
  \begin{align*}
    &\langle \CCR u,w\rangle = \sum_{f\in \mathcal{F}} \langle C_f u,w\rangle = \sum_{f\in \mathcal{F}} \langle C_f u,C_fw\rangle = \sum_{f\in \mathcal{F}} \langle u,C_fw\rangle = \langle u,\CCR w\rangle.
  \end{align*}
\end{proof}
\begin{theorem}[Ellipticity]
  The operator $\CCR$ satisfies~\ref{itm:ellipticity} with condition number
  \begin{align*}
   \operatorname{cond}_2(\CCR|_{\CR}) = \frac{\lambda_\textup{max}(\CCR|_{\CR})}{\lambda_\textup{min}(\CCR|_{\CR})} \leq \frac{d^2}{d+2}\qquad\text{for all }d\geq 2.
  \end{align*}
\end{theorem}
\begin{proof}\textit{Step 1 (Values $K_1$ and $K_2$).}
For a simplex $T\in \mathcal{T}$ let $x_0,\dots,x_d$ be its vertices and denote by $f_j$ the face opposite of vertex $x_j$ for all $j=0,\dots,d$. 
By $\psi_j \coloneqq \psi_{f_j}$ for all $j=0,\dots,d$ we denote the corresponding basis functions. 
We want to compute the constants $K_1,K_2$ such that, for all $v = \sum_{j=0}^d v_j \psi_j$ with $v_j\in \mathbb{R}$, 
  \begin{align}\label{eq:AimProofCR1}
    \sum_{j=0}^d \lVert v_j \psi_j\rVert_{2,T}^2 \leq K_1 \lVert v \rVert_{2,T}^2\quad \text{and}\quad \lVert v \rVert_{2,T}^2 \leq K_2 \sum_{j=0}^d \lVert v_j \psi_j\rVert_{2,T}^2. 
  \end{align} 
  With $\lambda_j$ denoting the barycentric coordinate corresponding to $x_j$, for $j = 0,\dots,d$, the basis function $\psi_j$ reads 
\begin{align*}
  \psi_j|_T = 1 - d\lambda_j\qquad\text{for all }j=0,\dots,d.
\end{align*}
Hence, the entries of the local mass matrix $M = (M_{j\ell})_{j,\ell=0}^d$ are given by 
\begin{align}\label{eq:LocMcr}
  \begin{aligned}
    M_{j\ell} &= \int_T \psi_{j}\psi_{\ell} \dx = \int_T
    (1-d\lambda_j) (1-d\lambda_\ell)\dx
    \\
    & = |T|\left(1 - 2d \frac{1}{d+1} + d^2
      \frac{(1+\delta_{j\ell})}{(d+2)(d+1)}\right)
    = |T|\, \frac{2-d + \delta_{j\ell}d^2}{(d+2)(d+1)}.
  \end{aligned}
\end{align}
Let $\identity\in \mathbb{R}^{(d+1)\times (d+1)}$ denote the identity matrix and set 
\begin{align*}
 N \coloneqq |T|\, \frac{2-d + d^2}{(d+2)(d+1)} \,\identity.
\end{align*}
The inequalities in \eqref{eq:AimProofCR1} are equivalent to 
  \begin{align*}
    \underline{v}^\top \underline{v} \leq K_1 \underline{v}^\top N^{-1} M \underline{v} \qquad\text{and}\qquad \underline{v}^\top N^{-1} M \underline{v}\leq K_2 \underline{v}^\top \underline{v}\qquad\text{for all }\underline{v} \in \mathbb{R}^{d+1}.
  \end{align*}
  This shows that the optimal values for $K_1$ and $K_2$ in \eqref{eq:AimProofCR1} are given by the inverse of the smallest eigenvalue and the largest eigenvalues of the matrix $N^{-1}M$, respectively.
  Let $\mathbf{1} \in \mathbb{R}^{(d+1)\times (d+1)}$ denote the matrix with all entries equal to one. Then one can compute that 
  \begin{align*}
  N^{-1} M = \frac{2-d}{2-d + d^2}\, \mathbf{1} + \frac{d^2}{2-d + d^2}\, \identity.
  \end{align*}  
   Since the spectrum of $\mathbf{1}$ and the one of $\identity$ equal $\lbrace 0 ,d+1\rbrace$ and $\lbrace 1\rbrace$, respectively, the spectrum of $N^{-1}M$ reads
  \begin{align*}
  \left\{  \frac{d^2}{2-d + d^2},   \frac{d^2 + (2-d)(d+1)}{2-d + d^2} \right\} = \left\{\frac{d^2}{2-d + d^2}, \frac{2+d}{2-d + d^2}\right\}. 
  \end{align*}
This leads to the constants 
  \begin{align*}
    K_1 = \frac{2-d + d^2}{2+d}
    \qquad\text{and}\qquad K_2 = \frac{d^2}{2-d + d^2}.
  \end{align*}
  \textit{Step 2 (Eigenvalue bounds for $\CCR$).}
  Let $v = \sum_{f\in \mathcal{F}} v_f \psi_f \in \CR$, then 
  \begin{align*}
    \norm{v}_2^2 = \sum_{f\in \mathcal{F}} \skp{v_f \psi_f}{v} = \sum_{f\in \mathcal{F}} \skp{v_f \psi_f}{C_f v} \leq \Big(\sum_{f\in \mathcal{F}} \lVert v_f \psi_f \rVert^2_2 \Big)^{1/2}\Big(\sum_{f\in \mathcal{F}} \lVert C_f v \rVert^2_2 \Big)^{1/2}.
  \end{align*}
  The application of the estimate in \eqref{eq:AimProofCR1} yields 
  \begin{align*}
    \norm{v}_2^2
    &\leq K_1 \sum_{f\in \mathcal{F}} \lVert C_f v \rVert^2_2 
      = K_1 \sum_{f\in \mathcal{F}} \skp{C_f v}{C_f v}               
      = K_1 \sum_{f\in \mathcal{F}} \skp{C_f v}{v}  = K_1 \skp{\CCR v}{v}.
  \end{align*}
  This proves that $1/ K_1 \leq \lambda_{\min}(\CCR|_{\CR})$.
  On the other hand, \eqref{eq:AimProofCR1} implies that
  \begin{align*}
    \norm{\CCR v}_2^2
    &= \lVert  \sum_{f\in \mathcal{F}} C_f v\rVert_2^2  \leq K_2
      \sum_{f\in\mathcal{F}} \norm{C_f v}_2^2
      = K_2 \sum_{f\in\mathcal{F}} \skp{C_f v}{C_f v}
    \\
    &
      = K_2 \sum_{f\in\mathcal{F}} \skp{C_f v}{v}  = K_2 \skp{\CCR v}{v} \leq K_2 \norm{\CCR v}_2 \norm{v}_2.
  \end{align*}
  This shows that $\lambda_{\max}(\CCR|_{\CR}) \leq K_2$. Finally, the claim follows from
\begin{align*}
&   \operatorname{cond}_2(\CCR|_{\CR}) = \frac{\lambda_\textup{max}(\CCR|_{\CR})}{\lambda_\textup{min}(\CCR|_{\CR})} \leq K_1K_2 = \frac{d^2}{d+2}.
\end{align*}
\end{proof}
We use a face based notion of neighbors such that the properties~\ref{itm:distance} and~\ref{itm:locality} hold by design of $\CCR$. 
\begin{definition}[Distance $\deltaCR$]
\label{def:MetricCR}
We call $T, T' \in \tria$ with $T \neq T'$ (face) neighbors, if they share a face $f \in \mathcal{F}$. 
Then the induced (geodesic) distance function $\deltaCR$ as in \ref{itm:distance} satisfies that $\deltaCR(T,T') = 1$ if and only if $T \cap T' \in \mathcal{F}$. 
\end{definition}

With this distance (which also enters Definition~\ref{def:grading-cont} of the grading) and with $\kappa \coloneqq \operatorname{cond}_2(\CCR|_{\CR})$, the combination of Proposition~\ref{prop:decay} with \eqref{eq:ConvSpeed2} and Theorem~\ref{thm:weighted-L2} shows the weighted $L^2$-stability for all weights $\rho$ with grading
\begin{align}\label{eq:MaxGradCR} 
\gamma_\rho < \frac{1}{q} =
  \frac{\sqrt{\kappa} + 1}{\sqrt{\kappa} - 1} = \frac{d + \sqrt{d+2}}{d
  - \sqrt{d+2}}.
\end{align} 
The distance~$\deltaCR$ as in Definition~\ref{def:MetricCR}
is larger than the distance~$\delta$ from Definition~\ref{def:Metric} since the corresponding notion of neighbors is stronger. 
This allows for smaller gradings in Definition~\ref{def:grading-cont}. 
In particular, we can define the following mesh size function $h\in \mathcal{L}_0^0(\mathcal{T})$. 
Let $\mathcal{T}$ be a triangulation resulting from successive application of the BiSec-MT algorithm to some fixed initial triangulation $\mathcal{T}_0$ of $\Omega \subset \mathbb{R}^d$, see Section \ref{subsec:RefStrat} for further details on BiSec-MT. 
As above we assume that the initial triangulation $\mathcal{T}_0$ satisfies the matching neighbor condition of Stevenson \cite{Stevenson08}. 
Then any simplex $T\in \mathcal{T}$ results from $\ell\in
\mathbb{N}$ recursive bisections of some simplex $T_0\in \mathcal{T}_0$. 
The number $\ell = \ell(T)$ is called level of $T\in
\mathcal{T}$. 
We define the mesh size function $h\in \mathcal{L}_0^0(\mathcal{T})$ with
\begin{align}\label{eq:meshSizeCR}
  h|_T \coloneqq 2^{-\ell(T)/d}\eqsim h_T \coloneqq \textup{diam}(T) \qquad\text{for all
  }T\in \mathcal{T}.
\end{align} 
\begin{lemma}[Grading for BiSec-MT]
The mesh size function $h$ as in \eqref{eq:meshSizeCR} (with respect to the distance~$\deltaCR$) has grading
\begin{align}\label{eq:GradingHcr}
\gamma_h = 2^{1/d}.
\end{align} 
\end{lemma}
\begin{proof}
The proof of \cite[Cor.\ 4.6]{Stevenson08} shows that the level of simplices $T,T' \in \mathcal{T}$ with $\deltaCR(T,T') = 1$ differs at most by one. 
Hence, the lemma follows from the definition of $h$. 
\end{proof}
\begin{corollary}[Weighted $L^p$-stability with BiSec-MT]\label{cor:LpStabCR}
Let $d\geq 2$, let $p \in [1, \infty]$ and let $\rho \in L^1(\Omega)$ be a weight with grading $\gamma_{\rho}$ such that 
\begin{align*} 
\gamma_\rho 2^{\left| \frac{1}{2} - \frac{1}{p}
  \right|} < \frac{1}{q} = \frac{d + \sqrt{d+2}}{d - \sqrt{d+2}}. 
\end{align*}
Then, for any mesh resulting from the successive application of BiSec-MT to an initial triangulation $\mathcal{T}_0$ as above one has that
\begin{align} \lVert \rho\,\QCR u \rVert_p \lesssim \lVert \rho u
  \rVert_p\qquad\text{for all }u \in L^p(\Omega). 
\end{align} 
The hidden constant solely depends on the initial triangulation $\tria_0$ and $\gamma_{\rho}$. 
In particular, $L^p$-stability holds for all $p\in [1,\infty]$ and all dimensions $d\leq 35$. 
\end{corollary}
\begin{proof}
The proof follows by application of Theorem~\ref{thm:weighted-Lp} and \eqref{eq:MaxGradCR}--\eqref{eq:GradingHcr}. 
\end{proof}
Let $\nablaNC$ denote the element-wise application of the
gradient. 
\begin{corollary}[Weighted Sobolev stability with BiSec-MT]\label{cor:W1pStabCR}
Let $d \geq 2$, let $p\in [1,\infty]$ and let $\rho \in L^1(\Omega)$ be a weight with grading $\gamma_{\rho}$ such that
\begin{align*} \gamma_\rho 2^{\frac{1}{d} + \left| \frac{1}{2} -
  \frac{1}{p} \right|} < \frac{1}{q} = \frac{d + \sqrt{d+2}}{d -
  \sqrt{d+2}}. 
\end{align*}
Then, for any mesh resulting from the successive application of BiSec-MT to an initial triangulation $\mathcal{T}_0$ as above one has that
\begin{align}
\lVert \rho\nablaNC \QCR u \rVert_p \lesssim \lVert \rho
  \nabla u \rVert_p\qquad\text{for all }u \in W^{1,p}(\Omega).
\end{align} 
The hidden constant solely depends on the initial triangulation $\tria_0$ and $\gamma_{\rho}$. 
In particular, for all dimensions $d\geq 2$ one has the $W^{1,2}$-stability
\begin{align} \lVert \nablaNC \QCR u \rVert_2 \lesssim \lVert \nabla u
  \rVert_2 \qquad\text{for all }u\in W^{1,2}(\Omega),
\end{align} 
and $W^{1,p}$-stability is available for all $p\in [1,\infty]$ and all dimensions $d\leq 32$. 
\end{corollary}
\begin{proof}
Let $\INC \colon W^{1,1}(\Omega)\to \CR$ defined by $\INC w = \sum_{f\in\mathcal{F}} \dashint_f w\,\mathrm{d}s\ \psi_f$ for all $w\in W^{1,1}(\Omega)$ be the non-conforming interpolation operator.
By \cite[Lem.\ 2]{OrtnerPraetorius11} this operator satisfies for all $w \in W^{1,1}(\Omega)$ and $p\in [1,\infty]$ the estimate 
  \begin{align*}
  \lVert w - \INC w\rVert_{p,T} \leq 2\, h_T\, \lVert \nabla w\rVert_{p,T}\qquad \text{and}\qquad \lVert \nabla \INC w\rVert_{p,T} \leq \lVert \nabla w\rVert_{p,T}.
  \end{align*}
Replacing the Scott--Zhang operator by $\INC$ in the proof of Theorem~\ref{thm:weighted-Sobolev} shows the claim. 
\end{proof}
\begin{remark}[Zero boundary values]
Similar arguments as the ones in Sections~\ref{sec:dirichl-bound-valu} and \ref{sec:stability-zero-val} extend the results of this section to the case of zero boundary conditions. 
\end{remark}

\section*{Acknowledgement}
We appreciate the truely helpful comments, the careful suggestions as well as the time invested by the referees. 
We would like to thank Fernando Gaspoz for his insights and valuable discussions on the grading of meshes generated by higher dimensional bisection algorithms.  
\bibliographystyle{amsalpha}

\bibliography{L2stab_H1proj}

\end{document}